\documentclass[11pt]{article}

\usepackage{ulem}
\usepackage{framed}

\usepackage{amsthm}
\usepackage{amssymb}
\usepackage{amscd}
\usepackage{amsmath}
\usepackage[all]{xy}
\usepackage[top=30truemm,bottom=30truemm,left=25truemm,right=25truemm]{geometry}
\usepackage{comment}
\usepackage{algorithmic,algorithm}
\usepackage{here}
\usepackage{graphicx}
\usepackage{color} %
\usepackage{enumerate} %
\usepackage{listliketab}
\usepackage{multirow}


\makeatletter
  
  \@addtoreset{algorithm}{subsection}
\makeatother

\theoremstyle{definition}
\newtheorem{pf}{{\it Proof}}

\theoremstyle{definition}

\theoremstyle{plain}
\newtheorem{theo}{Theorem}

\theoremstyle{plain}
\newtheorem{theor}{Theorem}

\theoremstyle{plain}

\theoremstyle{plain}
\newtheorem{propp}{Proposition}[section]

\theoremstyle{plain}
\newtheorem{thm}{Theorem}[subsection]

\theoremstyle{definition}
\newtheorem{ex}[thm]{Example}

\theoremstyle{definition}

\theoremstyle{definition}

\theoremstyle{definition}
\newtheorem{defin}[thm]{Definition}

\theoremstyle{definition}
\newtheorem{rem}[thm]{Remark}

\theoremstyle{plain}
\newtheorem{prop}[thm]{Proposition}

\theoremstyle{plain}
\newtheorem{lem}[thm]{Lemma}

\theoremstyle{plain}
\newtheorem{cor}[thm]{Corollary}

\theoremstyle{definition}

\theoremstyle{definition}

\theoremstyle{definition}

\theoremstyle{definition}

\theoremstyle{definition}
\newtheorem{prob}[thm]{Problem}

\numberwithin{equation}{subsection}

\def\qed{\hfill $\Box$}

\def\F{\mathbb{F}}

\def\deg{{\rm deg}}

\def\dim{{\rm dim}}

\def\mod{{\rm mod}}
\def\Im{{\rm Im}}
\def\Ker{{\rm Ker}}

\newcommand{\bbP}{{\operatorname{\bf P}}}

\newcommand{\Spec}{\operatorname{Spec}}

\newcommand{\Proj}{\operatorname{Proj}}

\newcommand{\Gal}{\operatorname{Gal}}

\newcommand{\GL}{\operatorname{GL}}

\newcommand{\SO}{\operatorname{SO}}

\newcommand{\gA}{\operatorname{A}}
\newcommand{\gB}{\operatorname{B}}
\newcommand{\gC}{\operatorname{C}}

\newcommand{\gH}{\operatorname{H}}

\newcommand{\gO}{\operatorname{O}}

\newcommand{\gT}{\operatorname{T}}
\newcommand{\gU}{\operatorname{U}}
\newcommand{\gV}{\operatorname{V}}
\newcommand{\gW}{\operatorname{W}}

\newcommand{\Aut}{\operatorname{Aut}}

\newcommand{\diag}{\operatorname{diag}}

\newcommand{\modulo}{\operatorname{mod}}

\newcommand{\fS}{{\mathfrak S}}
\newcommand{\Jac}{\operatorname{Jac}}

\makeatletter
\def\@seccntformat#1{\csname the#1\endcsname. }
\makeatother
\makeatletter
\renewcommand\section{\@startsection {section}{1}{\z@}%
 {-3.5ex \@plus -1ex \@minus -.2ex}%
 {2.3ex \@plus.2ex}%
 {\normalfont\large\bfseries}}
\makeatother

\begin{document}

\title
{\bf Superspecial curves of genus $4$ in small characteristic}
\author
{Momonari Kudo\thanks{Graduate School of Mathematics, Kyushu University.
E-mail: \texttt{m-kudo@math.kyushu-u.ac.jp}}
\ and Shushi Harashita\thanks{Graduate School of Environment and Information Sciences, Yokohama National University.
E-mail: \texttt{harasita@ynu.ac.jp}}}
\maketitle
\begin{abstract}
This paper contains a complete study of
superspecial curves of genus $4$ in characteristic $p\le 7$.
We prove that there does not exist a superspecial curve of genus $4$ in characteristic $7$.
This is a negative answer to the genus $4$ case of
the problem proposed by Ekedahl \cite{Ekedahl} in 1987.
This implies the non-existence of maximal curve of genus $4$ over $\F_{49}$,
which updates the table at {\tt manypoints.org}.
We give an algorithm to enumerate superspecial nonhyperelliptic curves
in arbitrary $p \ge 5$, and for $p\le 7$
we excute it with our implementation on a computer algebra system Magma.
Our result in $p=5$ re-proves
the uniqueness of maximal curves of genus $4$ over $\F_{25}$,
see \cite{FGT} for the original theoretical proof.

In Appendix, we present a general method determining Hasse-Witt matrices of curves which are complete intersections.


\end{abstract}

\section{Introduction}
Let $p$ be a rational prime.
Let $K$ be a perfect field of characteristic $p$.
Let $\overline K$ denote the algebraic closure of $K$.
By a curve, we mean a non-singular projective variety of dimension $1$.
A curve over $K$ is called {\it superspecial} if its Jacobian
is isomorphic to a product of supersingular elliptic curves over $\overline K$.
It is known that $C$ is superspecial if and only if the Frobenius on
$H^1(C,{\mathcal O}_C)$ is zero.

This paper concerns the enumeration of superspecial curves of genus $g=4$.
Our interest in the case of $g=4$ comes from the fact that
for $g\ge 4$ the dimension of the moduli space of curves of genus $g$ is strictly less than that of the moduli space of principally polarized abelian varieties of dimension $g$. This fact means that
the theory on abelian varieties is not so effective for
our purpose for $g\ge 4$.
In \cite{Ekedahl}, Theorem 1.1, Ekedahl proved that
if there exists a superspecial curve $C$ of genus $g$ in characteristic $p$,
then $2g \le p^2-p$, and $2g\le p-1$ if $C$ is hyperelliptic
and $(g,p)\ne (1,2)$.
In particular there is no superspecial curve
of genus $4$ in characteristic $\le 3$
and there is no superspecial hyperelliptic curve of genus $4$
in characteristic $\le 7$.
Hence in this paper we restrict ourselves to
the nonhyperelliptic case in $p\ge 5$.

Our main results in this paper are the following.

\begin{theo}\label{MainTheorem_intro}
Any superspecial curve of genus $4$ over $\F_{25}$ is $\F_{25}$-isomorphic to
\[
2yw + z^2=0,\qquad x^3 + a_1 y^3 + a_2 w^3 + a_3 zw^2=0
\]
in $\bbP^3$, where $a_1, a_2 \in \F_{25}^\times$ and $a_3\in \F_{25}$.
\end{theo}

By Theorem \ref{MainTheorem_intro}, we can give another proof of the uniqueness of maximal curves over $\mathbb{F}_{25}$ 
(Corollary \ref{MainCorollary} and Example \ref{CompleteRepresentativesChar5}), see \cite{FGT} for the original theoretical proof.

\begin{theo}\label{MainTheorem2_intro}
There is no superspecial curve of genus $4$ in characteristic $7$.
\end{theo}

Theorem \ref{MainTheorem2_intro} gives a negative answer to the genus $4$ case of
the problem proposed by Ekedahl  in 1987, see p.\ 173 of \cite{Ekedahl}.
Also this implies the non-existence of maximal curve of genus $4$ over $\F_{49}$,
which updated the table at {\tt manypoints.org}.
The site updates the upper and lower bounds of
$N_q(g)$ the maximal number of rational points on curves of genus $g$
over $\F_q$, after the paper \cite{GV} of van der Geer and van der Vlugt was published.
See Section 5.1 for the details of our contribution to the value of $N_{49} (4)$.
The authors learn much about this from E.\ W.\ Howe.

Here, we briefly describe our strategy to prove Theorems \ref{MainTheorem_intro} and \ref{MainTheorem2_intro}.
First we give a criterion for the superspeciality
of curves defined by two equations,
which is reminiscent of Yui's result \cite{Yui} on
hyperelliptic curves.
We regard a quadric as a (possibly degenerate) quadratic form $Q$.
Considering transformations by elements of the orthogonal group associated to $Q$,
we reduce parameters of cubic forms as much as possible.
After that, we enumerate cubic forms $P$ such that $V ( P, Q)$ are superspecial.
Our concrete algorithm of the enumeration is as follows.
We regard coefficients of $P$ as variables, and construct a multivariate system by our criterion for the superspeciality. 
We solve the system with the hybrid method \cite{BFP}.
In this paper, we are only interested in the solutions whose entries are lying in the ground field.
For each solution, we evaluate it to coefficients of $P$, and decide whether $V( P, Q)$ is non-singular or not with the Gr\"{o}bner basis computation.

We implemented\footnote{We implemented the algorithm on Magma V2.20-10 \cite{Magma}. For the source codes and log files, see the web page of the first author \cite{HPkudo}.} our algorithm over a computer algebra system Magma \cite{Magma}.
Note that our algorithm utilizes the Gr\"{o}bner basis computation to solve a system of algebraic equations derived from our criterion.
More concretely, our implementation adopts Magma's function \texttt{Variety}, which can compute all solutions of a zero-dimensional multivariate system over the ground field.
 
We here describe reasons why our computation has terminated in practical time.
As we mentioned above, we solve a system of algebraic equations (SA) in our enumeration algorithm. 
It is known that the number of indeterminants deeply affects the complexity of solving SA (see, e.g., \cite{Ayad}).
Thus, reducing indeterminants is quite important for our computation to terminate in practical time.

Another reason is to solve SA with optimal term ordering of indeterminants (in the graded reversible lexicographical order) and the hybrid technique proposed in \cite{BFP}.
The cost of solving SA with the Gr\"{o}bner basis computation depends on the number of solutions and term ordering.
Moreover, it is known that there is a tradeoff between an exhaustive search and Gr\"{o}bner bases techniques, see \cite[Section 3]{BFP}.
Hence considering such a tradeoff and deciding how many and which coefficients are regarded as indeterminants are important factors for our computation to terminate in practical time avoiding the out of memory errors.
From this, we heuristically decide how many and which coefficients are optimal to be regarded as indeterminants from experimental computations.
Consequently we have succeeded in finishing all the computation to guarantee the main theorems with a standard note PC.

This paper is organized as follows.
In Section \ref{section:2}, we collect some known facts on
curves of genus 4 and so on.
In Section \ref{sec:algorithm}, we give a method to obtain the Hasse-Witt matrix of a curve defined by two equations.
Section \ref{SectionReduction} is devoted to reduction of cubic forms by elements of the orthogonal group associated to $Q$.
In Section \ref{sec:main_results}, we prove our main theorems.
Our main theorems are reduced to some computational problems in Section \ref{subsec:sscurves}.
In order to solve the problems, we give in Section \ref{subsec:algorithm} an algorithm to enumerate superspecial curves based on the method in Section \ref{sec:algorithm}.
Then we solve in Section \ref{subsec:comp_result} the problems by executing our enumeration algorithm. 
We also show in Section \ref{subsec:imple} the timing data and sample codes of our computation on Magma.
In Section \ref{section:6}, we determine the automorphism group of the superspecial curve of genus $4$ in characteristic $5$ and enumerate
the isomorphism classes of superspecial curves over $\F_{25}$.
In Appendix A, we review a basic method to solve the radical membership problem for polynomial rings via the Gr\"{o}bner basis computation.
In Appendix B, we give a method to compute the Hasse-Witt matrix of a curve defined as a complete intersection via Koszul complex.
This is a generalization of the genus $4$ case method given in Section 3.1.

\subsection*{Acknowledgments}
The first author thanks Kazuhiro Yokoyama for helpful advice on the Gr\"{o}bner basis computation, proofs with computer algebra systems and their correctnesses.
The second author would like to thank Keiichi Gunji for his comments to the preliminary version of this paper.
The authors are very grateful to
Gerard van der Geer and Everett W. Howe,
who gave us helpful comments and suggestions
after we uploaded the first version of this paper to arXiv.
This work was supported by
JSPS Grant-in-Aid for Young Scientists (B) 21740006.

\section{Preliminaries}\label{section:2}
We review some facts on curves of genus $4$.

\subsection{Nonhyperelliptic curves of genus $4$}\label{CurvesGenus4}
Let $p$ be a rational prime.
Let $K$ be a perfect field of characteristic $p$.
Let $C$ be a nonhyperelliptic curve of genus $4$ over $K$.
The canonical divisor defines an embedding $C$ into $\bbP^3=\Proj(K[x,y,z,w])$.
It is known that $C$ is defined by an irreducible quadratic form $Q$ and
an irreducible cubic form $P$ in $x,y,z,w$ (cf. \cite{Har}, Chapter IV, Example 5.2.2).

We claim that it can be assumed that any coefficient of $Q$ and $P$
belongs to $K$. Let $\overline K$ be the algebraic closure of $K$.
Let $\Gamma_K$ denote the Galois group $\Gal(\overline K/K)$.
Since up to constant multiplication $Q$ is uniquely determined by $C$ (cf. \cite{Har}, Chapter IV, Example 5.2.2),
for any $\sigma\in\Gamma_K$ there is
$\lambda_\sigma\in \overline{K}^\times$ such that $Q^\sigma = \lambda_\sigma Q$. The map sending $\sigma$ to $\lambda_\sigma$ defines an element of
the Galois cohomology group
$H^1(\Gamma_K, \overline K^\times)$.
As this is trivial (Hilbert's theorem 90),
there is $a\in \overline K^\times$
such that $\lambda_\sigma = a^\sigma a^{-1}$ for all $\sigma\in \Gamma_K$.
Replacing $Q$ by $a^{-1} Q$,
we have $Q^\sigma = Q$ for all $\sigma\in \Gamma_K$.
From now on we assume that any coefficient of $Q$ belongs to $K$.

Choose $P$ such that $P$ and $Q$ determine $C$.
Since $C$ is defined over $K$, for any $\sigma\in\Gamma_K$ we have
\begin{equation}\label{P and its twist}
P^\sigma \equiv \lambda P \quad \mod \quad Q
\end{equation}
for some $\lambda\in \overline K^\times$.
(Otherwise $C$ is the intersection of $Q$, $P$, $P^\sigma$
with essentially different $P$ and $P^\sigma$. This is absurd.)

Let $m$ be the leading monomial of $Q$ with respect to a certain term order.
Let $V$ be 
the space 
of cubic forms over $K$ without terms containing $m$. Note $\dim_K V= 16$.
Taking modulo $Q$,
it suffices to consider $P$ belonging to 
$V_{\overline K}:=V\otimes_K{\overline K}$.
%
Then we have $P^\sigma \in V_{\overline K}$.
It follows from (\ref{P and its twist}) that $P^\sigma = \mu_\sigma P$
for some $\mu_\sigma \in \overline K^\times$.
By the same argument as in the second paragraph, we see
that $P$ belongs to $V$ after multiplying $P$ by a constant.

\begin{rem}\label{QuadFormOverFiniteField}
Note that a quadratic form $Q$ defining a curve $C$ of genus $4$
is possibly degenerate.
But its rank is greater than or equal to $3$,
since otherwise $Q$ is reducible.
If $K$ is a finite field of characteristic $p\ne 2$,
then the quadratic forms over $K$ are classified by their ranks and discriminants,
whence an element of $\GL_4(K)$ transforms $Q$
into $2xw+2yz$ or $2xw+y^2-\epsilon z^2$ with $\epsilon\not\in (K^\times)^2$
when $Q$ is non-degenerate,
resp. into $2yw - \epsilon z^2$ for a representative $\epsilon$ of $K^\times/(K^\times)^2$ when $Q$ is degenerate.
\end{rem}

\subsection{Superspecial curves and maximal curves}
It is known that maximal curves over $\F_{p^2}$
are superspecial.
For small $p$ a sort of the converse holds:
\begin{lem}\label{ExistenceMaximalCurve}
Assume that there exists a superspecial curve $C$ of genus $g$
over an algebraically closed field $k$ in characteristic $p$
with $g > \frac{p^2+1}{2p}$. 
Then there exists a maximal curve $X$ of genus $g$ over $\F_{p^2}$
such that $X_k:=X\times_{\Spec(\F_{p^2})} \Spec(k)$ is isomorphic to $C$.
\end{lem}
\begin{proof}
By the proof of \cite{Ekedahl}, Theorem 1.1, 
$C$ descends to a curve $X$ over $\F_{p^2}$ with 
$\sharp X(\F_{p^2}) = 1\pm 2gp+p^2$.
By the assumption, $1 - 2 g p+p^2$ is negative.
Hence
$\sharp X(\F_{p^2}) = 1 + 2 g p+p^2$ holds, i.e.,
$X$ is a maximal curve over $\F_{p^2}$.
\end{proof}
\section{Determining superspecialities}\label{sec:algorithm}

Let $p$ be a rational prime. 
Let $K$ be a perfect field of characteristic $p$.
In this section, we present a method to get the Hasse-Witt matrix of a curve in $\mathbf{P}^3$ defined by two equations, and give an algorithm to test whether the variety defined by given equations is non-singular or not.


\subsection{Hasse-Witt matrices of curves defined by two equations}\label{subsec:HWgenus4}

Given a graded module $M$ and an integer $\ell$, we denote by $M ( \ell )$ its $\ell$-twist given by $M (\ell)_t = M_{\ell + t}$.
We first show the following lemma, which gives graded free resolutions for ideals generated by two homogeneous polynomials (see also \cite[Chapter 1, Exercise 1]{Eisenbud}).


\begin{lem}\label{lem:1}
Let $S = K [ X_0 , \ldots , X_r ]$ be the polynomial ring with $r + 1$ variables over a field $K$ (not necessarily perfect nor of positive characteristic).
Let $f$ and $g$ be homogeneous polynomials with $\mathrm{gcd} ( f, g ) = 1$ in $S \smallsetminus \{ 0 \}$.
Put $c := \mathrm{deg} ( f )$ and $d := \mathrm{deg} ( g )$.
Then for every $n > 0$, the following sequence is an exact sequence of graded $S$-modules:
$$\xymatrix{
0 \ar[r]^(0.35){\varphi_3^{( n )}} & S \left( - ( c + d ) n \right) \ar[r]^(0.45){\varphi_2^{( n )}} & S ( - c n) \oplus S ( - d n ) \ar[r]^(0.75){\varphi_1^{( n )}} & S \ar[r]^(0.30){\varphi_0^{( n )}}  & S / \langle f^n , g^n \rangle_S \ar[r] & 0,
}$$
where $\varphi_0^{( n )}$ is the canonical homomorphism and the set of homomorphisms $\{ \varphi_2^{( n)}, \varphi_1^{( n )} \}$ is represented by the set of matrices $\{ A_2, A_1 \}$ as follows:
\begin{equation}
\left\{ A_2 = 
\left[
\begin{array}{cc}
g^n & - f^n
\end{array}
\right],
A_1 = 
\left[
\begin{array}{c}
f^n \\
g^n
\end{array}
\right]
\right\}.\nonumber
\end{equation}
\end{lem}

\begin{proof}
To simplify the notations, we denote by $\varphi_i$ the homomorphism $\varphi_i^{( n )}$ for each $0 \leq i \leq 3$ in this proof.
We show $\mathrm{Ker} \left( \varphi_1 \right) = \langle [ g^n, - f^n ] \rangle_S$, where $\langle [ g^n, - f^n ] \rangle_S$ denotes the (graded) submodule generated by $[ g^n, - f^n ]$ in $S ( - c n ) \oplus S ( - d n )$.
Let $[ h_1, h_2 ]$ be an arbitrary element in the submodule $\mathrm{Ker} \left( \varphi_1 \right) \subset S ( - c n) \oplus S ( - d n )$.
Then we have $\varphi_1 \left( [ h_1, h_2 ] \right) = f^n h_1 + g^n h_2 = 0$ and thus $f^n h_1 = - g^n h_2$.
Since $\mathrm{gcd} ( f, g ) = 1$, there exists $r \in S$ such that $h_1 = r g^n$ and $h_2 = - r f^n$.
Hence we have $[ h_1, h_2 ] = r \cdot [ g^n, - f^n ] \in \langle [ g^n, - f^n ] \rangle_S$.
Conversely, we have $\varphi_1 \left( [ g^n, - f^n ] \right) = f^n g^n - g^n f^n = 0$.
Thus it follows that $\mathrm{Ker} \left( \varphi_1 \right) = \langle [ g^n, - f^n ] \rangle_S$, i.e., $\mathrm{Ker} \left( \varphi_1 \right) = \mathrm{Im} \left( \varphi_2 \right) $. 
Clearly we have $\mathrm{Ker} \left( \varphi_0 \right) = \mathrm{Im} \left( \varphi_1 \right)$ and $\varphi_2$ (resp. $\varphi_0$) is injective (resp. surjective).
\end{proof}

\begin{rem}
In this paper, a homomorphism of graded $S$-modules is a graded homomorphism \textit{of degree zero} of graded $S$-modules.
In more detail, we say that a graded homomorphism 
\[
\varphi : \bigoplus_{j = 1}^{t} S(-d_j) \longrightarrow \bigoplus_{j = 1}^{t^{\prime}} S(- d_j^{\prime}) \ ; \ {\bf v} \mapsto {\bf v} A
\]
is \textit{of degree zero} if each $( k, \ell )$-entry $g_{k, \ell} \in S$ of the representation matrix $A = \left( g_{k, \ell} \right)_{k, \ell}$ is homogeneous of degree $d_k - d_{\ell}^{\prime}$.
\end{rem}

By Lemma \ref{lem:1}, the following lemma also holds.

\begin{lem}\label{lem:2}
Let $S = K [ X_0 , \ldots , X_r ]$ be the polynomial ring with $r + 1$ variables over a field $K$ (not necessarily perfect nor of positive characteristic).
Let $f$ and $g$ be homogeneous polynomials with $\mathrm{gcd} ( f, g ) = 1$ in $S \smallsetminus \{ 0 \}$.
Put $c := \mathrm{deg} ( f )$, $d := \mathrm{deg} ( g )$, $I := \langle f, g \rangle_S$ and $I_n := \langle f^n, g^n \rangle_S$ for $n > 0$.
Then for every $n > 0$, the following diagram of homomorphisms of graded $S$-modules commutes, and each horizontal sequence is exact:
$$\xymatrix{
0 \ar[r]^(0.35){\varphi_3^{( n )}} & S \left( - ( c + d ) n \right)  \ar[d]^{\psi_{2}} \ar[r]^(0.45){\varphi_{2}^{( n )}} &                      S ( - c n ) \oplus S ( - d n ) \ar[d]^{\psi_{1}} \ar[r]^(0.75){\varphi_{1}^{( n )}} & S  \ar[d]^{\psi_{0}} \ar[r]^(0.5){\varphi_0^{( n )}} & S / I_n                  \ar[d]^{\psi} \ar[r]    & 0            \\
0 \ar[r]^(0.35){\varphi_3^{( 1 )}} & S \left( - ( c + d ) \right)                                     \ar[r]^(0.45){\varphi_2^{( 1 )}} & S ( - c  ) \oplus S ( - d  )  \ar[r]^(0.75){\varphi_{1}^{( 1 )}} & S         \ar[r]^(0.5){\varphi_0^{( 1 )}} & S / I     \ar[r]    & 0            
}$$
where each $\varphi_i^{( m )}$ is defined as in Lemma \ref{lem:1}, $\psi_0$ is the identity map on $S$, and $\psi$ is the homomorphism defined by $h + I_n \mapsto h + I$.
In addition, the set of homomorphisms $\{ \psi_2, \psi_1 \}$ is represented by the set of matrices $\{ P_2, P_1 \}$ as follows:
\begin{equation}
\left\{ P_2 = 
\left[
\begin{array}{cc}
( f g )^{n-1}
\end{array}
\right],
P_1 = 
\left[
\begin{array}{cc}
f^{n-1} & 0\\
0 & g^{n-1}
\end{array}
\right]
\right\}.\nonumber
\end{equation}
\end{lem}

Let $K$ be a perfect field with $\mathrm{char} ( K )  = p > 0$ (e.g., $K = \mathbb{F}_{p^s}$ or $\overline{\mathbb{F}_{p^s}}$).
Let $\mathbf{P}^3 = \mathrm{Proj} ( S )$ denote the projective $3$-space for the polynomial ring $S := K [ x, y, z, w]$.
For a graded $S$-module $M$, let $\widetilde{M}$ denote the sheaf associated with $M$ on $\mathbf{P}^3$.
Now we describe a method to compute the Hasse-Witt matrix of the curve $C = V ( f, g ) \subset \mathbf{P}^3$ for given $p$ and homogeneous polynomials $f, g \in S$ with $\mathrm{gcd} ( f, g) = 1$ in $S$.
We use the same notations as in Lemma \ref{lem:2}, and take $n = p$.
Put $\varphi_i := \varphi_i^{(1)}$ and
\begin{equation}
\Phi_i:=\widetilde{\varphi_i},\quad \Phi_i^{(p)}:= \widetilde{\varphi_i^{(p)}},\quad \Psi:=\widetilde{\psi},\quad \mbox{and}\quad \Psi_i:=\widetilde{\psi_i}. \nonumber
\end{equation}
By Lemma \ref{lem:2}, the following diagram commutes:
$$\xymatrix{
H^1 \left( C, {\cal O}_C \right)                  \ar[dd]_{F^{\ast}} \ar[rr]^(0.5){\cong} \ar[rd]^{{(F_1 |_{C^p})}^{\ast}} &  & H^2 ( \mathbf{P}^3, \widetilde{I} )     \ar[r]^(0.45){\cong}              \ar[d]^{F_1^{\ast}}  & \mathrm{Ker} \left( H^3 ( \Phi_2 ) \right) \ar[d]^{F_1^{\ast}} \\
& H^1 \left( C^p, {\cal O}_{C^p} \right) \ar[r]^{\cong} \ar[ld]^{H^1 ( \Psi )} & H^2 ( \mathbf{P}^3, \widetilde{I_p} ) \ar[r]^(0.45){\cong} \ar[d]^{H^2 ( \Psi_0 )} & \mathrm{Ker} \left( H^3 ( \Phi_2^{( p )} ) \right) \ar[d]^{H^3 ( \Psi_2 )} \\
H^1 \left( C, {\cal O}_C \right)                  \ar[rr]^(0.5){\cong} &  & H^2 ( \mathbf{P}^3, \widetilde{I} ) \ar[r]^(0.45){\cong} & \mathrm{Ker} \left( H^3 ( \Phi_2 )  \right) 
}$$
where $F_1$ (resp. $F$) is the Frobenius morphism on $\mathbf{P}^3$ (resp. $C$) and $C^p :=  V \left( f^p, g^p \right)$.
The following proposition is viewed as a generalization of Proposition 4.21 in \cite[Chapter 4]{Har}.

\begin{prop}\label{prop:HW_general}
Let $K$ be a perfect field with $\mathrm{char} ( K )  = p > 0$.
Let $f$ and $g$ be homogeneous polynomials with $\mathrm{deg}( f ) = c$ and $\mathrm{deg} ( g ) = d$ such that $\mathrm{gcd}( f, g ) = 1$ in $S:=K [x,y,z,w]$, $c \leq 3$ and $d \leq 3$.
Let $C = V ( f, g )$ be the curve defined by the equations $f = 0$ and $g = 0$ in $\mathbf{P}^3$.
Write $(f g)^{p-1} =  \sum c_{i_1, i_2, i_3, i_4} x^{i_1} y^{i_2} z^{i_3} w^{i_4}$ and
\begin{equation}
\{ (k,\ell,m,n) \in ( \mathbb{Z}_{<0} )^4 : k + \ell + m + n = - c - d \} = \{ (k_1, \ell_1, m_1, n_1), \ldots , (k_r, \ell_r, m_r, n_r ) \}, \nonumber
\end{equation}
where we note that $r = \mathrm{dim}_K H^1 (C, \mathcal{O}_C)$.
Then the Hasse-Witt matrix of $C$ is given by
\begin{equation}
\left[
\begin{array}{ccc}
	c_{- k_1 p + k_1, - \ell_1 p + \ell_1, - m_1 p + m_1, - n_1 p + n_1} & \cdots & c_{- k_r p + k_1, - \ell_r p + \ell_1, - m_r p + m_1, - n_r p + n_1} \\
\vdots & & \vdots \\
	c_{- k_1 p + k_r, - \ell_1 p + \ell_r, - m_1 p + m_r, - n_1 p + n_r} & \cdots & c_{- k_r p + k_r, - \ell_r p + \ell_r, - m_r p + m_r, - n_r p + n_r}
\end{array}
\right]. \nonumber
\end{equation}
%
\end{prop}

\begin{proof}
By Lemma \ref{lem:2}, we have the following commutative diagram:
$$\xymatrix{
H^1 \left( C, \mathcal{O}_C \right) \ar[d]_{{(F_1 |_{C^p})}^{\ast}} \ar[rr]^(0.45){\cong} &  & \Ker \left( H^3 ( \Phi_2 ) \right) \ar[d]^{F_1^{\ast}} \ar[rr]^(0.4){\cong} & & H^3 \left( \mathbf{P}^3, \mathcal{O}_{\mathbf{P}^3} ( - c - d ) \right) \ar[d]^{F_1^{\ast}} \\
H^1 \left( C^p, \mathcal{O}_{C^p} \right) \ar[rr]^(0.45){\cong} \ar[d]_{H^1 ( \Psi )} & & \mathrm{Ker} \left( H^3 \left( \Phi_2^{( p )} \right) \right) \ar[rr] \ar[d]^{H^3 ( \Psi_2 )} & & H^3 \left( \mathbf{P}^3, \mathcal{O}_{\mathbf{P}^3} ( (- c - d) p ) \right) \ar[d]^{( f g )^{p-1}} \\
H^1 \left( C, \mathcal{O}_C \right)   \ar[rr]^(0.45){\cong} &  & \mathrm{Ker} \left( H^3 ( \Phi_2 )  \right) \ar[rr]^(0.4){\cong} & & H^3 \left( \mathbf{P}^3, \mathcal{O}_{\mathbf{P}^3} ( - c - d ) \right)
}$$
where $F_1$ (resp. $F$) is the Frobenius morphism on $\mathbf{P}^3$ (resp. $C$).
The $K$-vector space $H^3 \left( \mathbf{P}^3, {\cal O}_{\mathbf{P}^3}( - c - d ) \right)$ has a basis $\{ x^k y^{\ell} z^m w^n : (k,\ell,m,n) \in ( \mathbb{Z}_{<0} )^4,\ k + \ell + m + n = - c - d \}$.
For each $(k_i, \ell_i, m_i, n_i)$,
\begin{eqnarray}
(f g)^{p-1} \cdot F_1^{\ast} \left( x^{k_i} y^{\ell_i} z^{m_i} w^{n_i} \right) & = & (f g)^{p-1} \cdot ( x^{k_i p} y^{\ell_i p} z^{m_i p} w^{n_i p}) \nonumber \\
& = & \sum c_{i_1, i_2, i_3, i_4} x^{i_1 + k_i p} y^{i_2 + \ell_i p} z^{i_3 + m_i p} w^{i_4 + n_i p} \nonumber \\
& = & \sum_{j=1}^r c_{- k_i p + k_j, - \ell_i p + \ell_j, - m_i p + m_j, - n_i p + n_j} x^{k_j} y^{\ell_j} z^{m_j} w^{n_j}. \nonumber
\end{eqnarray}
Hence our claim holds.
\end{proof}

\begin{rem}
As in Proposition \ref{prop:HW_general}, we can compute the Hasse-Witt matrix algorithmically for more general cases, see \cite[Section 5]{Kudo} for details. 
\end{rem}

\begin{cor}\label{prop:HW}
Let $K$ be a perfect field with $\mathrm{char} ( K )  = p > 0$.
Let $f$ and $g$ be homogeneous polynomials with $\mathrm{deg}( f ) = 3$ and $\mathrm{deg} ( g ) = 2$ such that $\mathrm{gcd}( f, g ) = 1$ in $S:=K [x,y,z,w]$.
Let $C = V ( f, g )$ be the curve defined by the equations $f = 0$ and $g = 0$ in $\mathbf{P}^3$.
Then the Hasse-Witt matrix of $C$ is $0$ if and only if all the coefficients of
\begin{equation}
\begin{array}{cccc}
( x^2 y z w )^{p-1}, & x^{2 p-1} y^{p-2} z^{p-1} w^{p-1}, & x^{2 p-1} y^{p-1} z^{p - 2} w^{p -1}, &  x^{2 p -1} y^{p-1} z^{p - 1} w^{p -2}, \\
x^{p-2} y^{2 p-1} z^{p-1} w^{p-1}, & ( x y^2 z w )^{p-1} , & x^{p-1} y^{2 p-1} z^{p - 2} w^{p -1}, &  x^{p -1} y^{2 p-1} z^{p - 1} w^{p -2}, \\
x^{p-2} y^{p-1} z^{2 p - 1} w^{p -1}, & x^{p-1} y^{p-2} z^{2 p-1} w^{p-1}, & ( x y z^2 w )^{p-1} , &  x^{p -1} y^{p-1} z^{2 p - 1} w^{p -2}, \\
 x^{p -2} y^{p-1} z^{p - 1} w^{2 p -1}, & x^{p-1} y^{p-2} z^{p-1} w^{2 p-1}, & x^{p-1} y^{p-1} z^{p - 2} w^{2 p -1}, & ( x y z w^2 )^{p-1}
\end{array} \nonumber
\end{equation}
in $(f g)^{p-1}$ are equal to $0$. 
\end{cor}

\subsection{Determining non-singularity by Gr\"{o}bner basis computation}

Let $K$ be a field (not necessarily perfect nor of positive characteristic).
In this subsection, for given homogeneous polynomials $f_1, \ldots , f_t$ in $S:=K [ X_0, \ldots , X_r]$, we give an algorithm to test whether the variety $V ( f_1, \ldots , f_t )$ in $\mathbf{P}^r = \mathrm{Proj} (\overline{K}[X_0,\ldots , X_r])$ defined by the equations $f_1 = 0, \ldots , f_t = 0$ is non-singular or not. 

Let us first review a known property in Gr\"{o}bner basis theory:
Let $\succ$ be a term order on the monomials $X_0, \ldots , X_r$, $F$ an extension of the field $K$ and $T := F [ X_0, \ldots , X_r ]$ the polynomial ring with $r + 1$ variables over $F$.
Let $I := \langle f_1, \ldots , f_t \rangle_S$ (resp.\ $J:= \langle f_1, \ldots , f_t \rangle_{T}$) the ideal in $S$ (resp.\ $T$) generated by $f_1, \ldots , f_t$.
Then, for all $g \in S \subset T$, we have $\mathrm{NF}_{J, \succ} (g) = \mathrm{NF}_{I, \succ} (g) \in S$, where $\mathrm{NF}_{I, \succ} ( \cdot ) $ (resp.\ $\mathrm{NF}_{J, \succ} ( \cdot ) $) denotes the normal form function with respect to the ideal $I$ (resp.\ $J$) and the term order $\succ$.

\begin{lem}\label{lem:non-sing}
Let $S = K [X_0, \ldots , X_r]$ and $\overline{S} = \overline{K} [ X_0, \ldots , X_r ]$ denote the polynomial rings of $r+1$ variables over a field $K$ and its algebraic closure $\overline{K}$, respectively.
Let $f_1, \ldots , f_t$ be homogeneous polynomials in $S$. 
Let $J (f_1, \ldots , f_t )$ denote the set of all the minors of degree $r - \mathrm{dim} (V (f_1, \ldots , f_t))$ of the Jacobian matrix of $f_1, \ldots , f_t$.
Then the following are equivalent:
\begin{description}
	\item[{\rm (1)}] The variety $V ( f_1, \ldots , f_t )$ is non-singular. 
	\item[{\rm (2)}] For each $0 \leq i \leq r$, the variable $X_i$ belongs to the radical of $\langle J ( f_1, \ldots , f_t ), f_1, \ldots , f_t \rangle_S$.
\end{description}
\end{lem}

\begin{proof}
Put $S^{\prime} := K [X_0, \ldots , X_r, Y ]$ and $\left(\overline{S}\right)^{\prime} := \overline{K} [X_0, \ldots , X_r, Y]$, where $Y$ is a new variable.
Here we fix a term order $\succ$ on the monomials $X_0, \ldots , X_r$ and $Y$.
Note that $\succ$ can be viewed as a term order on the monomials $X_0, \ldots , X_r$.

(1) $\Rightarrow$ (2): Suppose that $V ( f_1, \ldots , f_t )$ is non-singular. 
	Then for each $0 \leq i \leq r$, the variable $X_i$ is in the radical of $\langle J ( f_1, \ldots , f_t ), f_1, \ldots , f_t \rangle_{\overline{S}} \subset \overline{S}$.
	Put $I := \langle J ( f_1, \ldots , f_t ), f_1, \ldots , f_t, 1 - Y X_i \rangle_{S^{\prime}}$ and $\overline{I}:= \langle J ( f_1, \ldots , f_t ), f_1, \ldots , f_t, 1 - Y X_i \rangle_{\left( \overline{S} \right)^{\prime}}$.
	By Proposition \ref{lem:radicalmenb}, we have $1 \in \overline{I}  \subset \left( \overline{S} \right)^{\prime}$, and thus $\mathrm{NF}_{\overline{I}, \succ} (1) = 0$.
	Here let $\mathrm{NF}_{I, \succ} ( \cdot ) $ denote the normal form function with respect to the ideal $I$ and the term order $\succ$.
	Since $\overline{I}$ is generated by elements in $S^{\prime}$, we also have $\mathrm{NF}_{\overline{I}, \succ} (1) = \mathrm{NF}_{I, \succ} (1)$.
	Thus it follows that $\mathrm{NF}_{I, \succ} (1) = 0$.
	Hence $X_i$ belongs to the radical of $\langle J ( f_1, \ldots , f_t ), f_1, \ldots , f_t \rangle_S \subset S$.

(2) $\Rightarrow$ (1): This claim clearly holds.
\end{proof}

In Algorithm \ref{alg:non-sing}, for given homogeneous polynomials $f_1, \ldots , f_t$ in $K [ X_0, \ldots , X_r]$, we give an algorithm to test whether the variety $V ( f_1, \ldots , f_t )$ in $\mathbf{P}^r = \mathrm{Proj} (\overline{K}[X_0,\ldots , X_r])$ (defined over $K$) is non-singular or not. 
Algorithm \ref{alg:non-sing} calls the sub-procedure function $\texttt{RadicalMembership}$ given in Appendix \ref{sec:radical}.

\begin{algorithm}[htb] %
\caption{$\texttt{DetermineNonSingularity} ( f_1, \ldots , f_t )$}
\label{alg:non-sing}
\begin{algorithmic}[1]
\REQUIRE{Homogeneous polynomials $f_1, \ldots , f_t$ in $S:=K [X_0, \ldots , X_r]$}
\ENSURE{``non-singular'' or ``singular''}
\STATE $nonsingularflag$ $\leftarrow$ $1$
\STATE $J (f_1, \ldots , f_t)$ $\leftarrow$ (the set of all the minors of degree $r - \mathrm{dim} (V (f_1, \ldots , f_t))$ of the Jacobian matrix of $f_1, \ldots , f_t$)
\STATE $I$ $\leftarrow$ $\langle J (f_1, \ldots , f_t), f_1, \ldots , f_t \rangle_S$
\FOR{$i=0$ \TO $r$}
	\IF{$\texttt{RadicalMembership} ( X_i, I)$ outputs ``$X_i \notin \sqrt{I}$''}
		\STATE $nonsingularflag$ $\leftarrow$ $0$
		\STATE \textbf{break} $i$
	\ENDIF
\ENDFOR
\IF{$nonsingularflag = 1$}
	\RETURN ``non-singular''
\ELSE
	\RETURN ``singular''
\ENDIF
\end{algorithmic}
\end{algorithm}

\section{Reduction of cubics by elements of orthogonal groups}\label{SectionReduction}
Let $p$ be a rational prime. Assume $p\ne 2$.
Let $K$ be a field of characteristic $p$.
Let $Q$ be $2xw + 2yz$ or $2xw + y^2 - \epsilon z^2$ for $\epsilon\in K^\times$ with $\epsilon\not\in (K^\times)^2$,
or $2yw - \epsilon z^2$ for a representative $\epsilon$ of $K^\times/(K^\times)^2$
as in Remark \ref{QuadFormOverFiniteField}.
Let $\varphi$ be the symmetric matrix associated to $Q$.
Consider the orthogonal group
$
\gO_\varphi(K) = \{g \in \GL_4(K) \mid {}^t g \varphi g = \varphi\}
$
and the orthogonal similitude group
\[
\tilde \gO_\varphi(K) = \{g \in \GL_4(K) \mid {}^t g \varphi g = \mu \varphi \text{ with } \mu\in K^\times\}.
\]
As $\mu$ is determined by $g$, we call $\mu=\mu(g)$ the similitude of $g$.
In this section, we study reductions of cubic forms over $K$ by elements of $\tilde \gO_\varphi(K)$.

\subsection{The orthogonal groups in the non-degenerate case}\label{non-degenerate case}
Consider the case that $\varphi$ is non-degenerate. There are two cases:
\[
\text{\bf (N1)}\quad
\begin{pmatrix}
0 & 0 & 0 & 1\\
0 & 0 & 1 & 0\\
0 & 1 & 0 & 0\\
1 & 0 & 0 & 0
\end{pmatrix},
\qquad \text{\bf (N2)}\quad
\begin{pmatrix}
0 & 0 & 0 & 1\\
0 & 1 & 0 & 0\\
0 & 0 & -\epsilon & 0\\
1 & 0 & 0 & 0
\end{pmatrix},
\]
where $\epsilon\in K^\times\smallsetminus(K^\times)^2$.
Let us review briefly
the structure of the orthogonal groups in each case.

\noindent{\bf (N1)} Put
$
\gT = \{\diag(a,b,b^{-1},a^{-1}) \mid a,b\in K^\times\}$ and
$\tilde \gT = \{\diag(a,b,cb^{-1},ca^{-1}) \mid a,b,c\in K^\times\}$,
\[
\gU = \left\{ \left. \begin{pmatrix}1&a&0&0\\0&1&0&0\\0&0&1&-a\\0&0&0&1\end{pmatrix}\begin{pmatrix}1&0&b&0\\0&1&0&-b\\0&0&1&0\\0&0&0&1\end{pmatrix}\right| a,b\in K\right\},
\]
\[
\gA = \left\{1_4, 
\begin{pmatrix}
1&0&0&0\\
0&0&1&0\\
0&1&0&0\\
0&0&0&1
\end{pmatrix}\right\},\quad
s_1 = \begin{pmatrix}
0&1&0&0\\
1&0&0&0\\
0&0&0&1\\
0&0&1&0
\end{pmatrix},\quad
s_2 = \begin{pmatrix}
0&0&1&0\\
0&0&0&1\\
1&0&0&0\\
0&1&0&0
\end{pmatrix}.
\]
Let $\gW:=\{1_4, s_1, s_2, s_1s_2\}$.
Put $\gB = \gA\gT\gU$ and $\tilde \gB = \gA\tilde \gT \gU$.
Recall the Bruhat decomposition
\[
\gO_\varphi(K)=\gB\gW\gU \quad\text{and}\quad \tilde\gO_\varphi(K)=\tilde \gB \gW \gU.
\]
This paper uses only the easy part $\tilde\gO_\varphi(K)\supset \tilde \gB \gW \gU$.
Remark that $\gA$ is introduced so that it is naturally isomorphic to $\gO_\varphi(K)/\SO_\varphi(K)$.


\noindent{\bf (N2)} Put
$\gH=\{\diag(a,1,1,a^{-1})\mid a \in K^\times\}$ and $\gA:=\{1_4, \diag(1,1,-1,1)\}$,
\[
\gU=\left\{\left.\begin{pmatrix}
1 & a & 0 & -a^2/2\\
0 & 1 & 0 & -a\\
0 & 0 & 1 & 0\\
0 & 0 & 0 & 1
\end{pmatrix}
\begin{pmatrix}
1 & 0 & b & b^2/(2\epsilon)\\
0 & 1 & 0 & 0\\
0 & 0 & 1 & b/\epsilon\\
0 & 0 & 0 & 1
\end{pmatrix}
\right| a,b \in K
\right\},
\]
\[
\tilde \gC=\left\{\left.
R(a,b):=\begin{pmatrix}
1 & 0 & 0 & 0\\
0 & a & \epsilon b & 0\\
0 & b & a & 0\\
0 & 0 & 0 & a^2-\epsilon b^2
\end{pmatrix} \right| a,b\in K\right\},\qquad  
\gW:=\left\{1_4, \begin{pmatrix}
0 & 0 & 0 & 1\\
0 & 1 & 0 & 0\\
0 & 0 & -1 & 0\\
1 & 0 & 0 & 0
\end{pmatrix}\right\}.
\]
Let $\gC = \{R(a,b)\in \tilde \gC \mid a^2-\epsilon b^2 = 1\}$
and put $\gT=\gH\gC$, $\tilde \gT= \gH\tilde \gC$,
$\gB=\gA\gT\gU$ and $\tilde \gB = \gA\tilde{\gT}\gU$.
We have the Bruhat decomposition
\[
\gO_\varphi(K) = \gB \gW \gU
\quad\text{and}\quad
\tilde \gO_\varphi(K) = \tilde \gB \gW \gU.
\]
We shall use only the easy part $\tilde\gO_\varphi(K)\supset \tilde \gB \gW \gU$.
Remark that $\gA$ is introduced so that it is naturally isomorphic to $\gO_\varphi(K)/\SO_\varphi(K)$.

The next is a key lemma for
the reduction of cubic forms in the case of (N2).
\begin{lem}\label{RepresentationRotationGroup}
Consider the natural representation of $\tilde \gC$
on the space $V$ of cubics in $y,z$ over $K$.
\begin{enumerate}
\item[\rm (1)]
$V$ is the direct sum of two subrepresentations $V_1:=\langle y(y^2-\epsilon z^2), z(y^2-\epsilon z^2)\rangle $
and $V_2:=\langle y(y^2+3\epsilon z^2), z(3y^2+\epsilon z^2)\rangle$.
\item[\rm (2)]
There are four $\tilde \gC$-orbits in $V_1$, which are the orbits
of $\delta y(y^2-\epsilon z^2)$ with $\delta \in \{0\} \cup K^\times/(K^\times)^3$.
\end{enumerate}
\end{lem}
\begin{proof}
(1) Straightforward.
(2) $R(a,b)$ sends $y(y^2-\epsilon z^2)$ to 
$(a^2-\epsilon b^2)\{a y(y^2-\epsilon z^2)+\epsilon bz(y^2-\epsilon z^2)\}$.
It suffices to show that for any $\alpha, \beta\in K$ there exist $a, b\in K$ and $\delta$ as above such that $\alpha = \delta (a^2-\epsilon b^2) a$
and $\beta = \epsilon\delta (a^2-\epsilon b^2)b$.
If $\beta = 0$, then we have $b=0$ and $\alpha = \delta a^3$,
whence there exist such $a,b$ and $\delta$.
If $\beta \ne 0$, we have $a = (\epsilon \alpha/\beta)b$
and therefore
$\alpha = \delta \lambda b^3$ for some $\lambda\in K$ not containing $a,b$.
There exist $b, \delta$ satisfying this equation. Putting $a = (\epsilon \alpha/\beta)b$, we have the desired $a,b$ and $\delta$.
\end{proof}

\subsection{The orthogonal groups in the degenerate case}\label{degenerate case}
Consider when $\varphi$ is degenerate:
\[
\varphi = \begin{pmatrix}
0 & 0 & 0 & 0\\
0 & 0 & 0 & 1\\
0 & 0 & -\epsilon & 0\\
0 & 1 & 0 & 0
\end{pmatrix}
\]
where $\epsilon\in K^\times$ is a representative of $K^\times/(K^\times)^2$.
Put $\gA:=\{1_4, \diag(1,1,-1,1)\}$,
\[
\gT:=\left\{\left. T(a):=\begin{pmatrix}
1&0&0&0\\
0&a&0&0\\
0&0& 1&0\\
0&0&0&a^{-1}\end{pmatrix} \right| a \in K^\times\right\},\quad
\gU := \left\{\left. U(a):=\begin{pmatrix}
1&0&0&0\\
0&1&a&a^2(2\epsilon)^{-1}\\
0&0&1&a \epsilon^{-1}\\
0&0&0&1\end{pmatrix} \right| a \in K\right\},
\]
\[
s:=\begin{pmatrix}
1&0&0&0\\
0&0&0&1\\
0&0&1&0\\
0&1&0&0\end{pmatrix},\quad
\gV=\left\{\left.\begin{pmatrix}
a&b&c&d\\
0&1&0&0\\
0&0&1&0\\
0&0&0&1\end{pmatrix} \right| a\in K^\times \text{ and } b, c, d\in K\right\}.
\]
Let $\tilde \gT := \{\diag(1,b,b,b) \mid b\in K^\times\} \gT$ and set $\gB:=\gA \gT \gU$ and $\tilde \gB := \gA \tilde \gT \gU$. We have
\begin{lem}
\begin{enumerate}
\item[\rm (1)] $\gO_\varphi(K) = (\gB \sqcup \gB s \gU) \gV$.
\item[\rm (2)]
$\tilde \gO_\varphi(K) = (\tilde \gB \sqcup \tilde \gB s \gU) \gV$.
\end{enumerate}
\end{lem}
\begin{proof}
It is straightforward that the right hand sides are contained in the left hand sides respectively.
Let $g\in \tilde \gO_\varphi(K)$. Write $\varphi=\begin{pmatrix}0 & 0 \\ 0 & \psi\end{pmatrix}$ and $g=\begin{pmatrix}\alpha & \beta\\ \gamma & \delta\end{pmatrix}$,
where $\psi$ and $\delta$ are square matrices  of degree $3$.
From the condition ${}^t g\varphi g = \mu \varphi$ with $\mu\in K^\times$, we have $\gamma=0$.
Hence
\[
g = \begin{pmatrix}1 & 0\\0&\delta\end{pmatrix}
\begin{pmatrix}\alpha & \beta\\0&1_3\end{pmatrix}
\quad\text{with}\quad\begin{pmatrix}\alpha & \beta\\0&1_3\end{pmatrix}\in \gV.
\]

(1) If $g \in \gO_\varphi(K)$, then
$\delta \in \gO_\psi(K)$.
It is also straightforward to check the Bruhat decomposition
that if $\delta$ is upper-triangular then $\begin{pmatrix}1 & 0\\0&\delta\end{pmatrix}\in \gB$ and otherwise $\begin{pmatrix}1 & 0\\0&\delta\end{pmatrix}\in \gB s \gU$.

(2) If $g \in \tilde \gO_\varphi(K)$, then $\delta\in\tilde \gO_\psi(K)$, i.e., ${}^t\delta\psi\delta=\mu(\delta)\psi$
with $\mu(\delta)\in K^\times$. Considering the determinants of the both sides,
we have $\det(\delta)^2 = \mu(\delta)^3$.
This implies $\mu(g)\in (K^\times)^2$, since $\mu(g)=\mu(\delta) = (\det(\delta)\mu(\delta)^{-1})^2 $. 
Choose $b\in K^\times$ with $b^2=\mu(g)$ and put $t=\diag(1,b,b,b)\in \tilde \gT$.
Clearly $t^{-1}g$ belongs to $\gO_\varphi(K)$. Then (2) follows from (1).
\end{proof}


\subsection{Reduction of cubic forms in the case of (N1)}
Assume $p\ne 2$.
Let $P=P(x,y,z,w)$ be an irreducible cubic form in $x,y,z,w$ over $K$.
Consider the case of (N1): $Q=2xw+2yz$.
We use the notation in Section \ref{non-degenerate case} (N1).
We consider a reduction of $P$
by an element of $\tilde \gO_\varphi(K)$
under the assumption:
\begin{quote}
{\bf (A1)}\quad $C$ has a $K$-rational point where 
\[
w = 1,\quad R_y(y,z):=\dfrac{\partial}{\partial y}P(-yz,y,z,1) \ne 0
\quad\text{and}\quad
R_z(y,z):=\dfrac{\partial}{\partial z}P(-yz,y,z,1) \ne 0.
\]
\end{quote}
We shall see that this assumption is satisfied if $C$ has sufficiently many $K$-rational points, see Remark \ref{Remark A1} (1) below.

\begin{enumerate}
\item[1.] Considering $\modulo Q$, we may consider only $P$ which does not have any term containing $xw$.
\begin{eqnarray}\label{N1_GeneralFormOfP}
P &= & a_1 x^3 + (a_2y+a_3z)x^2+(a_4y^2 + a_5yz + a_6z^2)x \nonumber\\
&& + a_7y^3 + a_8y^2 z + a_9y z^2 + a_{10}z^3\\
&& + (a_{11}y^2+a_{12}yz+a_{13}z^2)w + (a_{14}y+a_{15}z)w^2 + a_{16}w^3. \nonumber
\end{eqnarray}
\item[2.] For a rational point $(-bc, b, c, 1)$ of $C(K)$, we have an element of $\gO_\varphi(K)$
\[
\begin{pmatrix}
-bc & -b & -c & 1\\
b & 0 & 1 & 0\\
c & 1 & 0 & 0\\
1 & 0 & 0 & 0
\end{pmatrix},
\]
which transforms $P$ into a cubic, say $P'$.
Let $a_1',\ldots,a'_{16}$ be the coefficients of $P'$ as in \eqref{N1_GeneralFormOfP}.
One can check that the $x^3$-coefficient of $P'$ is $P(-bc,b,c,1)$
and the $x^2y$-coefficient of $P'$ is $R_z(b,c)$ 
and the $x^2z$-coefficient of $P'$ is $R_y(b,c)$. 
By assumption (A1), there exists $(b, c)$ such that $a'_1 = 0, a'_2\ne 0, a'_3\ne 0$.
Thus we may assume that $P$ has
$a_1 = 0, a_2\ne 0, a_3\ne 0$.
\item[3.] By the action of an element of $\gU$, we can eliminate the terms of $xy^2, xz^2$ from $P$.
Thus we may assume that $P$ has
$a_1 = 0, a_2\ne 0, a_3\ne 0, a_4=0, a_6=0$.
\item[4.] Composing some element ($y\mapsto cy, z\mapsto z/c$) of $\gT$ and some constant-multiplication to the whole $P$, we transform $P$ into a cubic where
\begin{enumerate}
\item[(i)]
the $y^3$-coefficient is $1$ and the $y^2z$-coefficient is $0$
or a representative of an element of $K^\times/(K^\times)^2$, or
\item[(ii)] the $y^3$-coefficient is $0$;
considering $y\leftrightarrow z$, we may assume that the $z^3$-coefficient is also $0$;
the $y^2z$-coefficient is $0$ or $1$;
and the $yz^2$-coefficient is $0$ or a representative of an element of $K^\times/(K^\times)^2$

\end{enumerate}
\item[5.] An element $(x\mapsto dx, w\mapsto w/d)$ of $\gT$ transforms $P$
into one whose coefficient of $z^2w$ is $0$ or $1$.
\end{enumerate}
\begin{lem}\label{ReductionLemmaN1}
Under assumption {\rm (A1)}, an element of $\tilde\gO_\varphi(K)$ transforms $P$ into
\begin{enumerate}
\item[\rm (i)]for $a_i\in K$
with $a_1\ne 0$, $a_2\ne 0$ 
and for $b_1\in\{0\}\cup K^\times/(K^\times)^2$ and $b_2\in\{0,1\}$,
\begin{eqnarray*}
&& (a_1 y + a_2 z)x^2  + a_3 yzx + y^3 + a_4 z^3 + b_1 y^2z + a_5 yz^2\\
&&  + (a_6 y^2 + a_7 yz + b_2 z^2)w + (a_8y + a_9z)w^2 + a_{10}w^3, \text{or}
\end{eqnarray*}
\item[\rm (ii)]for $a_i\in K$ with $a_1\ne 0$, $a_2\ne 0$ 
and for $b_1\in\{0,1\}$, $b_2\in\{0\}\cup K^\times/(K^\times)^2$ and $b_3\in\{0,1\}$,
\begin{eqnarray*}
&& (a_1 y + a_2 z)x^2  + a_3 yzx  + b_1 y^2z + b_2 yz^2\\
&&  + (a_4 y^2 + a_5 yz + b_3 z^2)w + (a_6 y + a_7 z)w^2 + a_8 w^3.
\end{eqnarray*}
\end{enumerate}
\end{lem}
\begin{rem}\label{Remark A1}
\begin{enumerate}
\item[(1)]
(A1) is satisfied at least if 
$|C(K)| > 36$. Indeed consider the cubic
\[
S_y: w\dfrac{\partial}{\partial y}P\left(-\frac{yz}{w},y,z,1\right) = P_x(x,y,z,w)(-z) + P_y(x,y,z,w)w=0.
\]
The number of points (with multiplicity) of $C\cap S_y$ is $3\cdot\deg(C)=18$.
Since $C\cap S_y \cap V(w)$ has at least $3$ points (with multiplicity),
the number of the points of $C$ with $w=1$ and
$R_y(y,z) = 0$
(resp. $R_z(y,z) = 0$) is at most $15$. 
Also the number of the points with $w=0$ is at most $6$ (with multiplicity).
\item[(2)]
(A1) is a technical assumption to reduce cases.
But the existence of a rational point is crucial to this reduction step
as it allows us to eliminate the term of $x^3$.
\end{enumerate}
\end{rem}

\subsection{Reduction of cubic forms in the case of (N2)}
Assume $p\ne 2, 3$.
Let $P=P(x,y,z,w)$ be an irreducible cubic form in $x,y,z,w$.
Consider the case of (N2): $Q=2xw + y^2 - \epsilon z^2$ with
$\epsilon \not\in(K^\times)^2$.
We use the notation in Section \ref{non-degenerate case} (N2).
We consider a reduction of $P$
by an element of $\tilde \gO_\varphi(K)$,
under the assumption
\begin{quote}
{\bf (A2)} $C$ has a $K$-rational point where 
\[
w = 1,\quad R_y(y,z):=\dfrac{\partial}{\partial y}P\left(-\dfrac{y^2 -\epsilon z^2}{2},y,z,1\right) \ne 0,
\quad
R_z(y,z):=\dfrac{\partial}{\partial z}P\left(-\dfrac{y^2 -\epsilon z^2}{2},y,z,1\right) \ne 0.
\]
\end{quote}
We shall see that this assumption is satisfied if $C$ has sufficiently many $K$-rational points, see Remark \ref{Remark A2} (1) below.
\begin{enumerate}
\item[1.] Considering $\modulo Q$, we may consider only $P$ which does not have any term containing $xw$,
see \eqref{N1_GeneralFormOfP}.
\item[2.] For a rational point $(-(b^2-\epsilon c^2)/2, b, c, 1)$ of $C(K)$, we have an element of $\gO_\varphi(K)$
\[
\begin{pmatrix}
-(b^2-\epsilon c^2)/2 & -b & \epsilon c & 1\\
b & 1 & 0 & 0\\
c & 0 & 1 & 0\\
1 & 0 & 0 & 0
\end{pmatrix},
\]
which transforms $P$ into a cubic, say $P'$.
Let $a_1',\ldots,a'_{16}$ be the coefficients of $P'$ as in \eqref{N1_GeneralFormOfP}.
The $x^3$-coefficient of $P'$ is $P(-(b^2-\epsilon c^2)/2,b,c,1)$
and the $x^2y$-coefficient of $P'$ is $R_y(b,c)$
and the $x^2z$-coefficient of $P'$ is $R_z(b,c)$.
By assumption (A2), there exists $(b, c)$ such that $a'_1 = 0, a'_2\ne 0, a'_3\ne 0$.
Thus we may assume that $P$ has $a_1 = 0, a_2\ne 0, a_3\ne 0$.
\item[3.] By the action of an element of $\gU$, we can transform $P$ into
a cubic form whose $x^1$-coefficient is a constant multiplication of 
$(y^2-\epsilon z^2)$, where $p\ne 3$ is necessary.
\item[4.] Considering the action of an element of $\tilde \gC$ and a constant-multiplication to the whole,
$P$ can be transformed into a cubic whose part of cubic form in $y,z$ is of the following form
\[
\alpha y(y^2-\epsilon z^2) + \beta y(y^2+3\epsilon z^2) + \gamma z(3y^2+\epsilon z^2)
\]
for some $\alpha\in\{0,1\}$ and $\beta,\gamma \in K$. Use Lemma \ref{RepresentationRotationGroup}.

\item[5.] An element $(x\mapsto cx, w\mapsto w/c)$ of $\gH$ transforms $P$ into a cubic
with $z^2w$-term $0$ or $1$.
\end{enumerate}

\begin{lem}\label{ReductionLemmaN2}
Under assumption {\rm (A2)}, an element of $\tilde\gO_\varphi(K)$ transforms $P$ into
\begin{eqnarray*}
&& (a_1 y + a_2 z)x^2  + a_3 (y^2-\epsilon z^2)x
+ b_1 y(y^2-\epsilon z^2) + a_4 y(y^2+3\epsilon z^2) + a_5 z(3y^2+\epsilon z^2)\\
&&  + (a_6 y^2 + a_7 yz + b_2 z^2)w + (a_8 y + a_9 z)w^2 + a_{10}w^3
\end{eqnarray*}
for $a_i\in K$ with $(a_1,a_2)\ne (0,0)$
and for $b_1,b_2\in\{0,1\}$.
\end{lem}

\begin{rem}\label{Remark A2}
\begin{enumerate}
\item[(1)] (A2) is satisfied provided
$|C(K)| > 37$.
Indeed consider the cubic
\[
S_y: w \frac{\partial}{\partial y}P\left(-\frac{y^2-\epsilon z^2}{2w},y,z,1\right)
= P_x(x,y,z,w)(-2y) + P_y(x,y,z,w)w = 0.
\]
As the number of $C\cap S_y$ is at most $18$,
the number of the points of $C$ with $w=1$ and
$R_y(y,z) = 0$
is at most $18$.
The same thing holds for
$R_z(y,z) = 0$.
Also the number of $K$-rational points with $w=0$ is at most one, since they satisfy $y^2-\epsilon z^2=0$ and therefore $y=z=0$.
\item[(2)]
For (A2) we remark the same thing as Remark \ref{Remark A1} (2).
\end{enumerate}
\end{rem}

\subsection{Reduction of cubic forms in the degenerate case}\label{Orthogonal group in the degenerate case}
Assume $p\ne 2,3$.
Let $P=P(x,y,z,w)$ be an irreducible cubic form in $x,y,z,w$.
The form $2yw-\epsilon z^2$ is equivalent to 
$2yw + z^2$ via $(x\mapsto x, y\mapsto y, z\mapsto z, w\mapsto -\epsilon w)$.
Set $Q = 2yw + z^2$.
We use the notation in Section \ref{degenerate case}
with $\epsilon=-1$.

If the coefficient of $x^3$ in $P$ is zero,
the point $(x,y,z,w)=(1,0,0,0)$ is on $C$ and $C$ is singular at the point.
Hence we may assume that the coefficient of $x^3$ in $P$ is not zero.

Let $K$ be a field with $\sharp K > 5$.
\begin{enumerate}
\item[1.] An element ($x\mapsto x+ay+bz+cw$) of $\gV$
transforms $P$  into a cubic without
terms of $x^2y$, $x^2z$, $x^2w$. We may assume that $P$ does not contain terms of $x^2y$, $x^2z$, $x^2w$.
\item[2.] Since $yw\equiv -2^{-1}z^2 \modulo Q$,
we may assume that $P$ does not have any term containing $yw$.
\item[3.]
We claim that there exists an element of $\gO_\varphi(K)$ stabilizing $x$
which transforms $P$ into $P'$
with non-zero term of $y^3$. Indeed
otherwise we may assume that the coefficients of $y^3, w^3$ are zero;
write $P(x=0)=\alpha y^2z + \beta yz^2 + \gamma z^3 + \lambda z^2w + \mu zw^2$;
then the $y^3$-coefficient of the cubic obtained by transforming $P$ by
an element of $s \gU s$, say $z\mapsto z-cy$, $w\mapsto w + cz - 2^{-1}c^2y$ for $c\in K$,
is
\[
-\alpha c + \beta c^2 -\gamma c^3 -(2^{-1}\lambda) c^4
-(4^{-1}\mu) c^5.
\]
If the values were zero for all $c\in K^\times$, then $\alpha,\beta,\gamma,\lambda,\mu=0$ must hold,
since $\sharp K^\times \ge 5$.
Hence $P'$ for some $c$ has non-zero $y^3$-term.
Then an element of $U$ transforms $P'$ into one without $y^2z$-term.
\item[4.] Composing an element ($y\mapsto cy, w\mapsto w/c$) of $\gT$ and a constant multiplication to the whole $P$,
the coefficients of $z^2w$ and $zw^2$ becomes $0, 1$.
\item[5.] By $x\mapsto d\cdot x$ for some $d\in K^\times$, we transform $P$ into a cubic
in which the leading coefficient of the coefficient $R$ of $x$ is one or $R=0$.
\end{enumerate}
\begin{lem}\label{ReductionLemmaDegenerate}
Assume $\sharp K > 5$.
An element of $\tilde\gO_\varphi(K)$ transforms $P$ into the following form
\begin{eqnarray*}
&& a_0x^3 + (a_1 y^2 + a_2 z^2 + a_3 w^2 + a_4 yz + a_5 zw)x\\
&& + a_6y^3 + a_7z^3 + a_8 w^3 + a_9yz^2 + b_1 z^2 w + b_2 zw^2
\end{eqnarray*}
for $a_i\in K$ and $b_1,b_2\in\{0,1\}$,
where $a_0, a_6\in K^\times$ and
the leading coefficient of $R:=a_1 y^2 + a_2 z^2 + a_3 w^2 + a_4 yz + a_5 zw$ is $1$ or $R=0$.
\end{lem}


\section{Main results}\label{sec:main_results}

Now we state our main results and prove them.
We put computational parts of the proofs together in Section \ref{subsec:comp_result}.
We choose an element $\epsilon$ of $\F_q$ for $q = 25$ and $49$ such that $-\epsilon$ is a generator of the cyclic group $\F_q^\times$, and fix it throughout this section.

\subsection{Superspecial curves}\label{subsec:sscurves}

\begin{theor}\label{MainTheorem}
Any superspecial curve of genus $4$ over $\F_{25}$ is $\F_{25}$-isomorphic to
\begin{equation}
2yw + z^2=0,\qquad x^3 + a_1 y^3 + a_2 w^3 + a_3 zw^2=0 \label{sscurve_25}
\end{equation}
in $\bbP^3$, where $a_1, a_2 \in \mathbb{F}_{25}^\times$ and $a_3\in \F_{25}$.
\end{theor}

\begin{proof}
As mentioned in Section 1, there is no superspecial hyperelliptic curve in characteristic $5$.
Let $C$ be a nonhyperelliptic curve of genus $4$ over $\mathbb{F}_{25}$.
We have seen in Section 2.1 that $C$ is defined
by a quadratic form $Q$ and a cubic form $P$ in $\F_{25} [x, y, z, w]$.
By Remark 2.1.1 and the first paragraph in Section 4.5,
we may assume that $Q$ is either of (N1) $2xw+2yz$ and (N2) $2xw+y^2-\epsilon z^2$
in the non-degenerate case and is $2yw+z^2$ in the degenerate case.


First we claim that there does not exist a superspecial curve $C=V(P,Q)$ if $Q$ is non-degenerate.
By Lemma \ref{ExistenceMaximalCurve}, if such a superspecial curve exists, then
it has a maximal curve over $\F_{25}$ as an $\F_{25}$-form, whose quadratic form is also non-degenerate.
Hence we may assume in addition that $C$ is a maximal curve.
Since a maximal curve has many rational points ($66$ rational points),
we can use the reduction in Section \ref{SectionReduction},
i.e., $P$ satisfies either of Lemma \ref{ReductionLemmaN1} (i) and (ii)
for (N1) case and Lemma \ref{ReductionLemmaN2} for (N2) case.
Now in each case we check that there does not exist a $P$ such that $C=V(P,Q)$ is a superspecial curve.
\begin{description}
	\item[(N1) {\rm (i):}] 
	If $P$ satisfies Lemma \ref{ReductionLemmaN1} (i), then $C=V(P,Q)$ is $\F_{25}$-isomorphic to $V(P',Q)$ for some
	\begin{eqnarray}
	P' & = & ( a_1 y + a_2 z) x^2 + a_3 y z x  + y^3 + a_4 z^3 + b_1 y^2 z + a_5 y z^2 \nonumber \\
	& & + ( a_6 y^2  + a_7 y z  + b_2 z^2 ) w + ( a_8  y + a_9 z ) w^2 + a_{10} w^3, \nonumber
	\end{eqnarray}
	where $a_1, a_2 \neq 0$, $b_1 \in \{ 0, 1, - \epsilon \}$ and $b_2 \in \{ 0, 1 \}$.
	By Proposition \ref{prop:N1(i)q25} in Section \ref{subsec:comp_result},
	there does not exist $( b_1, b_2, a_1, \ldots , a_{10} )$ such that $V (P', Q)$ is a superspecial curve.
	\item[(N1) {\rm (ii):}] If $P$ satisfies Lemma \ref{ReductionLemmaN1} (ii), then
	$C= V(P,Q)$ is $\F_{25}$-isomorphic to $V(P',Q)$ for some
	\begin{eqnarray}
	P' & = & ( a_1 y + a_2 z ) x^2 + a_3 y z x + b_1 y^2 z + b_2 y z^2  \nonumber \\
	& & + ( a_4 y^2  + a_5 y z + b_3 z^2) w + ( a_6 y  + a_7 z ) w^2 + a_8 w^3, \nonumber
	\end{eqnarray}
	where $a_1, a_2 \neq 0$, $b_1, b_3 \in \{ 0, 1 \}$ and $b_2 \in \{ 0, 1, - \epsilon \}$.
	By Proposition \ref{prop:N1(ii)q25} in Section \ref{subsec:comp_result},
	there does not exist $( b_1, b_2, b_3, a_1, \ldots , a_8 )$ such that $V (P', Q)$ is a superspecial curve.
	\item[(N2){\rm :}] Lemma \ref{ReductionLemmaN2} says that $C=V(P,Q)$ is $\F_{25}$-isomorphic to $V(P',Q)$ for some
	\begin{eqnarray}
	P' & = & ( a_1 y + a_2 z ) x^2 + a_3 (y^2 - \epsilon z^2) x + b_1 y ( y^2 - \epsilon z^2 ) + a_4 y ( y^2 + 3 \epsilon z^2 ) + a_5 z (3 y^2 + \epsilon z^2 ) \nonumber \\
	& & + ( a_6 y^2 + a_7 y z + b_2 z^2 ) w + ( a_8 y + a_9 z ) w^2 + a_{10} w^3, \nonumber
	\end{eqnarray}
	where $(a_1, a_2) \neq (0, 0)$ and $b_1, b_2 \in \{ 0, 1 \}$.
	By Proposition \ref{prop:N2q25} in Section \ref{subsec:comp_result}, 
	there does not exist $( b_1, b_2,  a_1, \ldots , a_{10} )$ such that $V (P', Q)$ is superspecial.
\end{description}

We next consider the degenerate case: $Q=2 y w + z^2$.
\begin{description}
	\item[Degenerate case{\rm :}] 
Lemma \ref{ReductionLemmaDegenerate} says that $C=V(P,Q)$ is $\F_{25}$-isomorphic to $V(P',Q)$ for some
	\begin{eqnarray}
	P' & = & a_0 x^3 + a_1 x y^2 + a_2 x z^2 + a_3 x w^2 + a_4 x y z  + a_5 x z w \nonumber \\
	& & + a_6 y^3 + a_7 z^3 + a_8 w^3 + a_9 y z^2 + b_1 z^2 w + b_2 z w^2, \nonumber
	\end{eqnarray}
	where $a_0, a_6 \neq 0$, $b_1, b_2 \in \{ 0, 1 \}$.
	By Proposition \ref{prop:Degenerate_q25} in Section \ref{subsec:comp_result}, $V (P', Q)$ is superspecial if and only if $a_0, a_6, a_8 \in \mathbb{F}_{25}^\times$, $b_2 \in \{ 0, 1 \}$, $a_i = 0$ for $i =1, \ldots , 5, 7, 9$ and $b_1 = 0$.
\end{description}
The theorem summarizes the above descriptions.
\end{proof}

Now we have a computational proof of the uniqueness of
superspecial curve of genus $4$ over an algebraically closed field, see \cite{FGT} for the original and theoretical proof.
\begin{cor}\label{MainCorollary}
All superspecial curves of genus $4$ in characteristic $5$ are isomorphic to each other over an algebraically closed field.
\end{cor}
\begin{proof}
It suffices to show that the curves in Theorem \ref{MainTheorem} are all isomorphic over an algebraically closed field.
By $(y\mapsto \lambda\mu y,z\mapsto \mu z, w\mapsto \frac{\mu}{\lambda} w)$ we have
\[
x^3 + a_1 \lambda^3 \mu^3 y^3 + a_2 \frac{\mu^3}{\lambda^3} w^3 + a_3 \frac{\mu^3}{\lambda^2} zw^2=0.
\]
There exists $(\lambda,\mu)$ such that $a_1 \lambda^3 \mu^3=1$ and $a_2 \frac{\mu^3}{\lambda^3}=1$.
One may consider only the following form
\[
C_\alpha:\quad x^3 + y^3 +  w^3 + \alpha zw^2=0.
\]
Finally we claim that there exists an element of $\gB s \gU$ transforming $C_0$ to $C_\alpha$,
by a computation with Gr\"obner basis, where indeterminates are some entries of matrices,
see Remark \ref{ComputationForUniquenessOfSSp} below.
\end{proof}

\begin{rem}\label{ComputationForUniquenessOfSSp}
One can verify the claim in the proof of Corollary \ref{MainCorollary}
from our computation programs over Magma \cite{Magma}, \cite{MagmaHP} and Maple \cite{MapleHP}.
For the programs with outputs, see the web page of the first author \cite{HPkudo}.
\end{rem}

\begin{theor}\label{MainTheorem2}
There is no superspecial curve of genus $4$ in characteristic $7$.
\end{theor}
\begin{proof}
By Lemma \ref{ExistenceMaximalCurve}, it suffices to show that
there does not exist a maximal curve of genus $4$ over $\F_{49}$.
As a maximal curve has many rational points ($106$ rational points),
we can use the reduction in Section \ref{SectionReduction}.
Similarly to the proof of Theorem \ref{MainTheorem},
the theorem follows from Propositions \ref{prop:N1(i)q49} -- \ref{prop:Degenerate_q49} in Section \ref{subsec:comp_result}.
\end{proof}

Here is a remark on $N_{49}(4)$,
where $N_q(g)$ denotes
the maximum number of rational points of curves over $\F_q$ of genus $g$.
Before this paper, it was known that $102 \le N_{49}(4) \le 106$.
Note that $106$ is the Hasse-Weil bound,
and the lower bound is due to E. W. Howe, 
see \cite{Howe}, {\sc{Table}} 3.
Theorem B
(the non-existence of maximal curves over $\F_{49}$ of genus $4$)
says that $N_{49}(4) \le 105$.
But the authors learned from E. W. Howe
that we can say much more.
He told us that
combining a result in Howe-Lauter \cite{Howe-Lauter} and Theorem B,
we have:

\begin{cor}
$N_{49}(4) = 102$.
\end{cor}
\begin{proof}
It remains to show that
there does not exist a genus-$4$ curve $C$ over $\F_{49}$
with $|C(\F_{49})|=103,104$ or $105$.
For a curve over $\F_q$ of genus $g$,
its defect is defined to be the difference between
the Hasse-Weil-Serre bound $q+1+g[2\sqrt{q}]$ and
the number of $\F_q$-rational points on the curve.
Serre has proved that any curve over a finite field can not have defect $1$ if its genus $>2$,
see \cite{Serre}, Thm 2 (1) on p.~15 and Theorem on p.~24.
Hence in our case $|C(\F_{49})|=105$ is not possible.
Also genus-$4$ curves over $\F_{49}$
can not have defect $2$ or defect $3$:
the reason is Howe-Lauter \cite{Howe-Lauter}, Theorem 3.1.
To check actually the non-existence of such curves, we used
Magma programs (\verb-IsogenyClasses.magma-)
associated with \cite{Howe-Lauter}, found at
\verb-http://alumnus.caltech.edu/~however/papers/paper35.html-;
more concretely executed
\verb-isogeny_classes(49,4,N)- for $N=103$ and $104$
after loading \verb-IsogenyClasses.magma-.
\end{proof} 

\subsection{An algorithm to enumerate superspecial curves of genus $4$}\label{subsec:algorithm}

In this subsection, we give an algorithm to enumerate superspecial curves of genus $4$.
Recall from Section 2.1 that a nonhyperelliptic curve of genus $4$ is defined by a quadratic form $Q$ and a cubic form $P$ in $x, y, z, w$.
Given $Q$, as in Section 4, $P$ is transformed into
\begin{equation}
\sum_{i=1}^t a_i p_i + \sum_{j=1}^u b_j q_j \label{input_P}
\end{equation}
for some cubics $p_i$'s and $q_j$'s.
Here $a_i$'s and $b_j$'s are exact values in $K$ determined from $Q$ and certain conditions of $P$.
Then our aim is to compute all $( a_1, \ldots , a_t, b_1, \ldots , b_u )$ such that $C = V( P, Q)$ is superspecial.
We describe our main algorithm for the enumeration.

\paragraph{Main Algorithm:}
Let $\mathcal{M}$ be the set of the $16$ monomials given in Corollary \ref{prop:HW}.
We fix a quadratic form $Q$ over $K:= \mathbb{F}_q$.
Given cubics $p_1, \ldots , p_t$, $q_1, \ldots , q_u$ and some conditions of coefficients, we enumerate all cubic forms of the forms \eqref{input_P} such that the curves $C = V ( P, Q )$ are superspecial.

\begin{enumerate}\setcounter{enumi}{0}
	\item (Deciding inputs of Algorithm \ref{alg:enume}) Choose $1 \leq s \leq t$ and indices $i_1, \ldots , i_s$, and then regard $a_{i_1}, \ldots , a_{i_s}$ as variables.
	The remaining part $(a_{j_1},\ldots,a_{j_{t^{\prime}}})$  ($\{j_1,\ldots,j_{t^{\prime}}\} = \{1,\ldots,t \}\setminus \{i_1,\ldots,i_s\}$) runs through a subset $\mathcal{A}$ of $(\mathbb{F}_q)^{t-s}$, which is	determined in each case.
	\item (Algorithm \ref{alg:enume}) For each $\left( a_{j_1}, \ldots , a_{j_{t^{\prime}}} \right) \in \mathcal{A}$, take the following procedures:
	\begin{enumerate}
		\item Compute the coefficients of the monomials of $\mathcal{M}$ in $(P Q)^{p-1}$.
		We denote by $\mathcal{S}$ the set of the coefficients.
		Note that $P \in \mathbb{F}_q [ a_{i_1}, \ldots , a_{i_s} ] [ x, y, z, w]$ in this step and that $\mathcal{S} \subset \mathbb{F}_q [ a_{i_1}, \ldots , a_{i_s} ]$.
		\item Solve the system of algebraic equations $f ( a_{i_1}, \ldots , a_{i_s} ) = 0$ for all $f \in \mathcal{S}$ over $\mathbb{F}_q$.
		\item For each solution $( a_{i_1}, \ldots , a_{i_s})$, decide whether $C = V ( P, Q)$ is non-singular or not.
	\end{enumerate}
\end{enumerate}
We give in Algorithm \ref{alg:enume} our main algorithm in an algorithmic format.

\begin{algorithm}[htb] %
\caption{$\texttt{EnumerateSSpCurves} ( Q, P, ( i_1, \ldots , i_s ), \mathcal{A}, p, \succ)$}
\label{alg:enume}
\begin{algorithmic}[1]
\REQUIRE{A quadratic form $Q$ in $S = K [x, y, z, w]$, a cubic form $P$ of the form \eqref{input_P} in $K [a_1, \ldots , a_t] [x,y,z,w]$, a tuple $( i_1, \ldots , i_s )$ with $1 \leq i_1 < \cdots < i_s \leq t$, a subset $\mathcal{A}$ of $K^{t-s}$, the characteristic $p$ of $K$, and a term order $\succ$ on $a_{i_1}, \ldots , a_{i_t}$}
\ENSURE{A family $\mathcal{P}$ of cubics $P$ such that the curves $C = V ( P, Q )$ are superspecial}
\STATE $\mathcal{P}$ $\leftarrow$ $\emptyset$
\STATE $\mathcal{M}$ $\leftarrow$ the set of the $16$ monomials given in Corollary \ref{prop:HW}
\STATE $t^{\prime}$ $\leftarrow$ $t-s$; Write $\{ 1, \ldots , t \} \smallsetminus \{ i_1, \ldots , i_s \} = \{ j_1, \ldots , j_{t^{\prime}} \}$
\FOR{$\left( a_{j_1}, \ldots , a_{j_{t^{\prime}}} \right) \in \mathcal{A}$}
	\STATE Substitute $\left( a_{j_1}, \ldots, a_{j_{t^{\prime}}} \right)$ to $P$ /* Keep $a_{i_1}, \ldots , a_{i_s}$ being indeterminates*/
	\STATE $h$ $\leftarrow$ $( P Q )^{p-1}$
	\STATE $\mathcal{S}$ $\leftarrow$ $\emptyset$
	\FOR{$x^k y^{\ell} z^m w^n \in \mathcal{M}$}
		\STATE $f \left( a_{i_1}, \ldots, a_{i_s} \right)$ $\leftarrow$ the coefficient of $x^k y^{\ell} z^m w^n$ in $h$
		\STATE $\mathcal{S}$ $\leftarrow$ $\mathcal{S} \cup \{ f \left( a_{i_1}, \ldots, a_{i_s} \right) \}$
	\ENDFOR
	\STATE $I$ $\leftarrow$ the ideal $\langle \mathcal{S} \rangle \subset K [ a_{i_1}, \ldots, a_{i_s} ]$
	\STATE Solve the system $f \left( a_{i_1}, \ldots, a_{i_s} \right) = 0$ for all $f \in \mathcal{S}$ over $K$ by some known algorithm with $\succ$
	\STATE $V$ $\leftarrow$ $V(I) = \{ ( a_{i_1}, \ldots , a_{i_s} ) \in K^s : f \left( a_{i_1}, \ldots, a_{i_s} \right) = 0 \mbox{ for all } f \in \mathcal{S} \}$
	\IF{$V \neq \emptyset$}
		\FOR{$\left( a_{i_1}, \ldots, a_{i_s} \right) \in V$}
			\STATE Substitute $\left( a_{i_1}, \ldots, a_{i_s} \right)$ to $P$ /* Then $P \in K[x,y,z,w]$ */
			\IF{$\texttt{DetermineNonSingularity} ( P, Q )$ outputs ``non-singular''}
				\STATE $\mathcal{P}$ $\leftarrow$ $\mathcal{P} \cup \{ P \}$
			\ENDIF
		\ENDFOR
	\ENDIF
\ENDFOR
\RETURN $\mathcal{P}$
\end{algorithmic}
\end{algorithm}

\begin{ex}\label{ex:N2}
In the case of {\bf (N2)} for $q=49$ in Section \ref{SectionReduction}, the fixed quadratic form is $Q = 2 x w + y^2 - \epsilon z^2$ and by Lemma \ref{ReductionLemmaN2} the transformed cubic form is
\begin{eqnarray}
P &=& ( a_1 y + a_2 z ) x^2 + a_3 (y^2 - \epsilon z^2) x + b_1 y ( y^2 - \epsilon z^2 ) + a_4 y ( y^2 + 3 \epsilon z^2 ) + a_5 z (3 y^2 + \epsilon z^2 ) \nonumber \\
& & + ( a_6 y^2 + a_7 y z + b_2 z^2 ) w + ( a_8 y + a_9 z ) w^2 + a_{10} w^3, \nonumber
\end{eqnarray}
which has $12$ independent coefficients, where $( a_1, a_2 ) \neq (0, 0)$ and $b_1, b_2 \in \{ 0, 1 \}$.
In this case, we put $t:=10$, $u=2$ and
\begin{eqnarray}
\{ p_1, \ldots , p_t \}  &:=& \{ y x^2, z x^2, (y^2 - \epsilon z^2) x , y ( y^2 + 3 \epsilon z^2 ), z (3 y^2 + \epsilon z^2 ), y^2 w, y z w, y w^2, z w^2, w^3 \}, \nonumber \\
\{ q_1, \ldots , q_u \} &:=& \{ y ( y^2 - \epsilon z^2 ) , z^2 w \}. \nonumber 
\end{eqnarray}
For each $( b_1, b_2 )$, we conduct the following procedures:
\begin{enumerate}
	\item Put $s := 7$, and $(i_1, \ldots , i_7) := ( 4, 5, 6, 7, 8, 9, 10)$ (we regard the $7$ coefficients $a_4$, $a_5$, $a_6$, $a_7$, $a_8$, $a_9$, $a_{10}$ as variables).
	We set $\mathcal{A}:= \{ ( a_1, a_2, a_3 ) \in (\mathbb{F}_{49})^3 : (a_1, a_2 ) \neq (0,0) \}$.
	\item For each $( a_1 , a_2, a_3 ) \in \mathcal{A}$, 
	\begin{enumerate}
		\item Compute the coefficients of the monomials of $\mathcal{M}$ in $(P Q)^{p-1}$.
		We denote by $\mathcal{S}$ the set of the coefficients.
		Note that $P \in \mathbb{F}_{49} [ a_4, a_5, a_6, a_7, a_8, a_9, a_{10}] [ x, y, z, w]$ in this step and that $\mathcal{S} \subset \mathbb{F}_{49} [ a_4, a_5, a_6, a_7, a_8, a_9, a_{10}]$.
		\item Solve the system of algebraic equations $f ( a_4, a_5, a_6, a_7, a_8, a_9, a_{10} ) = 0$ for all $f \in \mathcal{S}$ over $\mathbb{F}_{49}$ via the Gr\"{o}bner basis computation.
		\item For each solution $( a_4, a_5, a_6, a_7, a_8, a_9, a_{10} )$, decide whether $C = V ( P, Q)$ is non-singular or not by Algorithm \ref{alg:non-sing}.
	\end{enumerate}
\end{enumerate}
For the term order we adopt, see the proof of Proposition \ref{prop:N2q49} in Section \ref{subsec:comp_result}.
\end{ex}

\subsection{Computational parts of our proofs of the main theorems}\label{subsec:comp_result}
To prove Theorems \ref{MainTheorem} and \ref{MainTheorem2}, we in this subsection give computational results obtained by our implementation in a computer algebra system (for details on the implementation, see Section \ref{subsec:imple}).
We execute Main Algorithm in Section \ref{subsec:algorithm} in each case, and use the notation in the algorithm.


\subsubsection{Case of (N1) (i) for $q = 25$}

\begin{prop}\label{prop:N1(i)q25}
Consider the quadratic form $Q = 2 x w + 2 y z$ and
\begin{eqnarray}
P &=& ( a_1 y + a_2 z) x^2 + a_3 y z x  + y^3 + a_4 z^3 + b_1 y^2 z + a_5 y z^2 \nonumber \\
& & + ( a_6 y^2  + a_7 y z  + b_2 z^2 ) w + ( a_8  y + a_9 z ) w^2 + a_{10} w^3, \nonumber
\end{eqnarray}
where $a_1, a_2 \neq 0$, $b_1 \in \{ 0, 1, - \epsilon \}$ and $b_2 \in \{ 0, 1 \}$.
Then there does not exist $( b_1, b_2, a_1, \ldots , a_{10} )$ such that $V (P, Q)$ is superspecial.
\end{prop}

\begin{proof}
We set $t = 10$, $u = 2$ and
\begin{eqnarray}
\{ p_1, \ldots , p_t \} & = & \{ y x^2, z x^2, y z x, z^3, y z^2,  y^2 w, y z w, y w^2, z w^2, w^3 \}, \nonumber \\
\{ q_1, \ldots , q_u \} & = & \{ y^2 z, z^2 w \}. \nonumber 
\end{eqnarray}
For each $( b_1, b_2 )$, we conduct the following procedures:
\begin{enumerate}
	\item Put $s := 8$, and $(i_1, \ldots , i_8) := ( 3, 4, 5, 6, 7, 8, 9, 10)$ (we regard the $8$ coefficients $a_3$, $a_4$, $a_5$, $a_6$, $a_7$, $a_8$, $a_9$, $a_{10}$ as variables).
	For computations concerned with Gr\"{o}bner bases over $\mathbb{F}_{25} [ a_3, a_4, a_5, a_6, a_7, a_8, a_9, a_{10}]$, we adopt the grevlex order with
\[
a_{10} \prec a_9 \prec a_4 \prec a_8 \prec a_7 \prec a_5 \prec a_6 \prec a_3.
\]
	For $\mathbb{F}_{25} [ x, y, z, w]$, we adopt the grevlex order with $w \prec z \prec y \prec x$.
	Put $\mathcal{A} = \mathbb{F}_{25}^{\times} \times \mathbb{F}_{25}^{\times}$.
	\item Execute Algorithm \ref{alg:enume}. More concretely, we proceed the following three steps for each $( a_1 , a_2 ) \in \mathcal{A}$: 
	\begin{enumerate}
		\item Compute the coefficients of the monomials of $\mathcal{M}$ in $(P Q)^{p-1}$.
		We denote by $\mathcal{S}$ the set of the coefficients.
		Note that $P \in \mathbb{F}_{25} [ a_3, a_4, a_5, a_6, a_7, a_8, a_9, a_{10}] [ x, y, z, w]$ in this step and that $\mathcal{S} \subset \mathbb{F}_{25} [ a_3, a_4, a_5, a_6, a_7, a_8, a_9, a_{10} ]$.
		\item Solve the system of algebraic equations $f ( a_3, a_4, a_5, a_6, a_7, a_8, a_9, a_{10} ) = 0$ for all $f \in \mathcal{S}$ over $\mathbb{F}_{25}$ via the Gr\"{o}bner basis computation.
		\item For each solution $( a_3, a_4, a_5, a_6, a_7, a_8, a_9, a_{10} )$, decide whether $C = V ( P, Q)$ is non-singular or not by Algorithm \ref{alg:non-sing}.
	\end{enumerate}
\end{enumerate}
It follows from the outputs of our computation that there does not exist $( b_1, b_2, a_1, \ldots , a_{10} )$ such that $V (P, Q)$ is superspecial.
\end{proof}

Note in this case that the number of possible $(b_1, b_2, a_1, a_2)$, the coefficients which we run, is $3 \cdot 2 \cdot 24 \cdot 24 = 3456$.

\subsubsection{Case of (N1) (ii) for $q = 25$}

\begin{prop}\label{prop:N1(ii)q25}
Consider the quadratic form $Q = 2 x w + 2 y z$ and
\begin{eqnarray}
P &=& ( a_1 y + a_2 z ) x^2 + a_3 y z x + b_1 y^2 z + b_2 y z^2  \nonumber \\
& & + ( a_4 y^2  + a_5 y z + b_3 z^2) w + ( a_6 y  + a_7 z ) w^2 + a_8 w^3, \nonumber
\end{eqnarray}
where $a_1, a_2 \neq 0$, $b_1, b_3 \in \{ 0, 1 \}$ and $b_2 \in \{ 0, 1, - \epsilon \}$.
Then there does not exist $( b_1, b_2, b_3, a_1, \ldots , a_8 )$ such that $V (P, Q)$ is superspecial.
\end{prop}

\begin{proof}
We set $t = 8$, $u = 3$ and
\begin{eqnarray}
\{ p_1, \ldots , p_t \} & = & \{ y x^2, z x^2, y z x, y^2 w, y z w, y w^2, z w^2, w^3 \}, \nonumber \\
\{ q_1, \ldots , q_u \} & = & \{ y^2 z, y z^2, z^2 w \}. \nonumber 
\end{eqnarray}
For each $( b_1, b_2, b_3 )$, we conduct the following procedures:
\begin{enumerate}
	\item Put $s := 6$, and $(i_1, \ldots , i_6) := ( 3, 4, 5, 6, 7, 8)$ (we regard the $6$ coefficients $a_3$, $a_4$, $a_5$, $a_6$, $a_7$, $a_8$ as variables).
In this case we adopt the grevlex order with
	\[
	a_8 \prec a_7 \prec a_6 \prec a_5 \prec a_4 \prec a_3.
	\]
	For $\mathbb{F}_{25} [ x, y, z, w]$, we adopt the grevlex order with $w \prec z \prec y \prec x$.
	Put $\mathcal{A} = \mathbb{F}_{25}^{\times} \times \mathbb{F}_{25}^{\times}$.
	\item For each $( a_1 , a_2 ) \in \mathcal{A}$, as in Case {\bf (N1)} (i), we enumerate $( a_3, a_4, a_5, a_6, a_7 , a_8 )$ such that $C$ is superspecial by Algorithm \ref{alg:enume}.
\end{enumerate}
It follows from the outputs of our computation that there does not exist $( b_1, b_2, b_3, a_1, \ldots , a_8 )$ such that $V (P, Q)$ is superspecial.
\end{proof}

The number of possible $(b_1, b_2, b_3, a_1, a_2)$ is $2 \cdot 3 \cdot 2 \cdot 24 \cdot 24 = 6912$.

\subsubsection{Case of (N2) for $q = 25$}

\begin{prop}\label{prop:N2q25}
Consider the quadratic form $Q = 2 x w + y^2 - \epsilon z^2$ and
\begin{eqnarray}
P &=& ( a_1 y + a_2 z ) x^2 + a_3 (y^2 - \epsilon z^2) x + b_1 y ( y^2 - \epsilon z^2 ) + a_4 y ( y^2 + 3 \epsilon z^2 ) + a_5 z (3 y^2 + \epsilon z^2 ) \nonumber \\
& & + ( a_6 y^2 + a_7 y z + b_2 z^2 ) w + ( a_8 y + a_9 z ) w^2 + a_{10} w^3, \nonumber
\end{eqnarray}
where $(a_1, a_2) \neq (0, 0)$ and $b_1, b_2 \in \{ 0, 1 \}$.
Then there does not exist $( b_1, b_2, a_1, \ldots , a_{10} )$ such that $V (P, Q)$ is superspecial.
\end{prop}

\begin{proof}
We set $t = 10$, $u = 2$ and
\begin{eqnarray}
\{ p_1, \ldots , p_t \} & = & \{ y x^2, z x^2, ( y^2 - \epsilon z^2) x, y ( y^2 + 3 \epsilon z^2 ), z (3 y^2 + \epsilon z^2 ), y^2 w , y z w, y w^2, z w^2, w^3 \}, \nonumber \\
\{ q_1, \ldots , q_u \} & = & \{ y ( y^2 - \epsilon z^2 ), z^2 w \}. \nonumber 
\end{eqnarray}
For each $( b_1, b_2 )$, we conduct the following procedures:
\begin{enumerate}
	\item Put $s := 8$, and $(i_1, \ldots , i_8) := (  3, 4, 5, 6, 7, 8, 9, 10 )$ (we regard the $8$ coefficients $a_3$, $a_4$, $a_5$, $a_{6}$, $a_{7}$, $a_{8}$, $a_{9}$, $a_{10}$ as variables).
	In this case we adopt the grevlex order with
\[
a_{10} \prec a_9 \prec a_8 \prec a_7 \prec a_6 \prec a_5 \prec a_4 \prec a_3.
\]
	For $\mathbb{F}_{25} [ x, y, z, w]$, we adopt the grevlex order with $w \prec z \prec y \prec x$.
	Put $\mathcal{A} = ( \mathbb{F}_{25} \times \mathbb{F}_{25} ) \smallsetminus \{ (0,0) \}$.
	\item For each $( a_1 , a_2 ) \in \mathcal{A}$, as in Case {\bf (N1)} (i), we enumerate $( a_3, a_4, a_5 , a_6, a_7, a_8, a_9, a_{10} )$ such that $C$ is superspecial by Algorithm \ref{alg:enume}.
\end{enumerate}
It follows from the outputs of our computation that there does not exist $( b_1, b_2, a_1, \ldots , a_{10} )$ such that $V (P, Q)$ is superspecial.
\end{proof}

The number of possible $(b_1, b_2, a_1, a_2 )$ is $2 \cdot 2 \cdot (25^2 - 1 ) = 2496$.

\subsubsection{Degenerate case for $q=25$}

\begin{prop}\label{prop:Degenerate_q25}
Consider the quadratic form $Q = 2 y w + z^2$ and
\begin{eqnarray}
P &=& a_0 x^3 + ( a_1 y^2 + a_2 z^2 + a_3 w^2 + a_4 y z  + a_5 z w ) x \nonumber \\
& & + a_6 y^3 + a_7 z^3 + a_8 w^3 + a_9 y z^2 + b_1 z^2 w + b_2 z w^2, \nonumber
\end{eqnarray}
where $a_0, a_6 \neq 0$, $b_1, b_2 \in \{ 0, 1 \}$.
Then $V (P, Q)$ is superspecial if and only if $a_0, a_6, a_8 \in \mathbb{F}_{25}^\times$, $b_2 \in \{ 0, 1 \}$, $a_i = 0$ for $i =1, \ldots , 5, 7, 9$ and $b_1 = 0$.
\end{prop}

\begin{proof}
We set $t = 10$, $u = 2$ and
\begin{eqnarray}
\{ p_1, \ldots , p_t \} & = & \{ x^3, x y^2, x z^2, x w^2, x y z, x z w, y^3, z^3, w^3, y z^2 \}, \nonumber \\
\{ q_1, \ldots , q_u \} & = & \{ z^2 w, z w^2 \}. \nonumber 
\end{eqnarray}
For each $( b_1, b_2)$, we conduct the following procedures:
\begin{enumerate}
	\item Put $s := 7$, and $(i_1, \ldots , i_7) := ( 2, 3, 4, 5, 7, 8, 9)$ (we regard the $7$ coefficients $a_2$, $a_3$, $a_4$, $a_5$, $a_7$, $a_8$, $a_9$ as variables).
	In this case we adopt the grevlex order with 
	\[
	a_8 \prec a_7 \prec a_9 \prec a_3 \prec a_5 \prec a_2 \prec a_4.
	\]
	For $\mathbb{F}_{25} [ x, y, z, w]$, we adopt the grevlex order with $w \prec z \prec y \prec x$.
	Put $\mathcal{A} = \mathbb{F}_{25}^{\times} \times \mathbb{F}_{25} \times \mathbb{F}_{25}^{\times}$.
	\item For each $( a_0, a_1, a_6)  \in \mathcal{A}$, as in Case {\bf (N1)} (i), we enumerate $( a_2, a_3, a_4, a_5, a_7, a_8, a_9 )$ such that $C$ is superspecial by Algorithm \ref{alg:enume}.
\end{enumerate}
It follows from the outputs of our computation that $V (P, Q)$ is superspecial if and only if $a_0, a_6, a_8 \in \F_{25}^\times$, $b_2 \in \{ 0, 1 \}$, $a_i = 0$ for $i =1, \ldots , 5, 7, 9$ and $b_1 = 0$.
\end{proof}

The number of possible $(b_1, b_2, a_0, a_1, a_6)$ is $2 \cdot 2 \cdot 24 \cdot 25 \cdot 24 = 57600$.

\subsubsection{Case of (N1) (i) for $q = 49$}

\begin{prop}\label{prop:N1(i)q49}
Consider the quadratic form $Q = 2 x w + 2 y z$ and
\begin{eqnarray}
P &=& ( a_1 y + a_2 z) x^2 + a_3 y z x  + y^3 + a_4 z^3 + b_1 y^2 z + a_5 y z^2 \nonumber \\
& & + ( a_6 y^2  + a_7 y z  + b_2 z^2 ) w + ( a_8  y + a_9 z ) w^2 + a_{10} w^3, \nonumber
\end{eqnarray}
where $a_1, a_2 \neq 0$, $b_1 \in \{ 0, 1, - \epsilon \}$ and $b_2 \in \{ 0, 1 \}$.
Then there does not exist $( b_1, b_2, a_1, \ldots , a_{10} )$ such that $V (P, Q)$ is superspecial.
\end{prop}

\begin{proof}
We set $t = 10$, $u = 2$ and
\begin{eqnarray}
\{ p_1, \ldots , p_t \} & = & \{ y x^2, z x^2, y z x, z^3, y z^2,  y^2 w, y z w, y w^2, z w^2, w^3 \}, \nonumber \\
\{ q_1, \ldots , q_u \} & = & \{ y^2 z, z^2 w \}. \nonumber 
\end{eqnarray}
For each $( b_1, b_2 )$, we conduct the following procedures:
\begin{enumerate}
	\item Put $s := 7$, and $(i_1, \ldots , i_7) := ( 4, 5, 6, 7, 8, 9, 10)$ (we regard the $7$ coefficients $a_4$, $a_5$, $a_6$, $a_7$, $a_8$, $a_9$, $a_{10}$ as variables).
	In this case we adopt the grevlex order with
\[
a_{10} \prec a_9 \prec a_4 \prec a_8 \prec a_7 \prec a_5 \prec a_6.
\]
	For $\mathbb{F}_{49} [ x, y, z, w]$, we adopt the grevlex order with $w \prec z \prec y \prec x$.
	Put $\mathcal{A} = \mathbb{F}_{49}^{\times} \times \mathbb{F}_{49}^{\times} \times \mathbb{F}_{49}$.
	\item For each $( a_1 , a_2,  a_3 ) \in \mathcal{A}$, as in Case {\bf (N1)} (i) for $q = 25$, we enumerate $( a_4, a_5, a_6, a_7, a_8, a_9, a_{10} )$ such that $C$ is superspecial by Algorithm \ref{alg:enume}.
\end{enumerate}
It follows from the outputs of our computation that there does not exist $( b_1, b_2, a_1, \ldots , a_{10} )$ such that $V (P, Q)$ is superspecial.
\end{proof}

The number of possible $(b_1, b_2, a_1, a_2, a_3)$ is $3 \cdot 2 \cdot 48 \cdot 48 \cdot 49 = 677376$.

\subsubsection{Case of (N1) (ii) for $q = 49$}

\begin{prop}\label{prop:N1(ii)q49}
Consider the quadratic form $Q = 2 x w + 2 y z$ and
\begin{eqnarray}
P &=& ( a_1 y + a_2 z ) x^2 + a_3 y z x + b_1 y^2 z + b_2 y z^2  \nonumber \\
& & + ( a_4 y^2  + a_5 y z + b_3 z^2) w + ( a_6 y  + a_7 z ) w^2 + a_8 w^3, \nonumber
\end{eqnarray}
where $a_1, a_2 \neq 0$, $b_1, b_3 \in \{ 0, 1 \}$ and $b_2 \in \{ 0, 1, - \epsilon  \}$.
Then there does not exist $( b_1, b_2, b_3, a_1, \ldots , a_8 )$ such that $V (P, Q)$ is superspecial.
\end{prop}

\begin{proof}
We set $t = 8$, $u = 3$ and
\begin{eqnarray}
\{ p_1, \ldots , p_t \} & = & \{ y x^2, z x^2, y z x, y^2 w, y z w, y w^2, z w^2, w^3 \}, \nonumber \\
\{ q_1, \ldots , q_u \} & = & \{ y^2 z, y z^2, z^2 w \}. \nonumber 
\end{eqnarray}
For each $( b_1, b_2, b_3 )$, we conduct the following procedures:
\begin{enumerate}
	\item Put $s := 6$, and $(i_1, \ldots , i_6) := ( 3, 4, 5, 6, 7, 8)$ (we regard the $6$ coefficients $a_3$, $a_4$, $a_5$, $a_6$, $a_7$, $a_8$ as variables).
	In this case we adopt the grevlex order with
\[
a_8 \prec a_7 \prec a_6 \prec a_5 \prec a_4 \prec a_3.
\]
	For $\mathbb{F}_{49} [ x, y, z, w]$, we adopt the grevlex order with $w \prec z \prec y \prec x$.
	Put $\mathcal{A} = \mathbb{F}_{49}^{\times} \times \mathbb{F}_{49}^{\times}$.
	\item For each $( a_1 , a_2 ) \in \mathcal{A}$, as in Case {\bf (N1)} (i) for $q = 25$, we enumerate $( a_3, a_4, a_5, a_6, a_7, a_8 )$ such that $C$ is superspecial by Algorithm \ref{alg:enume}.
\end{enumerate}
It follows from the outputs of our computation that there does not exist $( b_1, b_2, b_3, a_1, \ldots , a_8 )$ such that $V (P, Q)$ is superspecial.
\end{proof}

The number of possible $(b_1, b_2, b_3, a_1, a_2)$ is $2 \cdot 3 \cdot 2 \cdot 48 \cdot 48 = 27648$.

\subsubsection{Case of (N2) for $q = 49$}

\begin{prop}\label{prop:N2q49}
Consider the quadratic form $Q = 2 x w + y^2 - \epsilon z^2$ and
\begin{eqnarray}
P &=& ( a_1 y + a_2 z ) x^2 + a_3 (y^2 - \epsilon z^2) x + b_1 y ( y^2 - \epsilon z^2 ) + a_4 y ( y^2 + 3 \epsilon z^2 ) + a_5 z (3 y^2 + \epsilon z^2 ) \nonumber \\
& & + ( a_6 y^2 + a_7 y z + b_2 z^2 ) w + ( a_8 y + a_9 z ) w^2 + a_{10} w^3, \nonumber
\end{eqnarray}
where $(a_1, a_2) \neq (0,0)$ and $b_1, b_2 \in \{ 0, 1 \}$.
Then there does not exist $( b_1, b_2, a_1, \ldots , a_{10} )$ such that $V (P, Q)$ is superspecial.
\end{prop}

\begin{proof}
We set $t = 10$, $u = 2$ and
\begin{eqnarray}
\{ p_1, \ldots , p_t \} & = & \{ y x^2, z x^2, ( y^2 - \epsilon z^2) x, y ( y^2 + 3 \epsilon z^2 ), z (3 y^2 + \epsilon z^2 ), y^2 w , y z w, y w^2, z w^2, w^3 \}, \nonumber \\
\{ q_1, \ldots , q_u \} & = & \{ y ( y^2 - \epsilon z^2 ), z^2 w \}. \nonumber 
\end{eqnarray}
For each $( b_1, b_2 )$, we conduct the following procedures:
\begin{enumerate}
	\item Put $s := 7$, and $(i_1, \ldots , i_7) := ( 4, 5, 6, 7, 8, 9, 10)$ (we regard the $7$ coefficients $a_4$, $a_5$, $a_6$, $a_7$, $a_8$, $a_9$, $a_{10}$ as variables).
	In this case we adopt the grevlex order with
\[
a_{10} \prec a_9 \prec a_8 \prec a_7 \prec a_6 \prec a_5 \prec a_4.
\]
	For $\mathbb{F}_{49} [ x, y, z, w]$, we adopt the grevlex order with $w \prec z \prec y \prec x$.
	Put $\mathcal{A} = \{ (a_1, a_2, a_3 ) \in \mathbb{F}_{49} \times \mathbb{F}_{49} \times \mathbb{F}_{49} : (a_1, a_2 ) \neq (0,0) \}$.
	\item For each $( a_1 , a_2, a_3 ) \in \mathcal{A}$, as in Case {\bf (N1)} (i) for $q=25$, we enumerate $( a_4, a_5 , a_6, a_7, a_8, a_9, a_{10} )$ such that $C$ is superspecial by Algorithm \ref{alg:enume}.
\end{enumerate}
It follows from the outputs of our computation that there does not exist $( b_1, b_2, a_1, \ldots , a_{10} )$ such that $V (P, Q)$ is superspecial.
\end{proof}

The number of possible $(b_1, b_2, a_1, a_2, a_3 )$ is $2 \cdot 2 \cdot (49^2 - 1) \cdot 49= 470400$.

\subsubsection{Degenerate case for $q=49$}

\begin{prop}\label{prop:Degenerate_q49}
Consider the quadratic form $Q = 2 y w + z^2$ and
\begin{eqnarray}
P &=& a_0 x^3 + a_1 x y^2 + a_2 x z^2 + a_3 x w^2 + a_4 x y z  + a_5 x z w \nonumber \\
& & + a_6 y^3 + a_7 z^3 + a_8 w^3 + a_9 y z^2 + b_1 z^2 w + b_2 z w^2, \nonumber
\end{eqnarray}
where $a_0, a_6 \neq 0$, $b_1, b_2 \in \{ 0, 1 \}$.
Then there does not exist $( b_1, b_2, a_0, \ldots , a_9 )$ such that $V (P, Q)$ is superspecial.
\end{prop}

\begin{proof}
We set $t = 10$, $u = 2$ and
\begin{eqnarray}
\{ p_1, \ldots , p_t \} & = & \{ x^3, x y^2, x z^2, x w^2, x y z, x z w, y^3, z^3, w^3, y z^2 \}, \nonumber \\
\{ q_1, \ldots , q_u \} & = & \{ z^2 w, z w^2 \}. \nonumber 
\end{eqnarray}
For each $( b_1, b_2)$, we conduct the following procedures:
\begin{enumerate}
	\item Put $s := 7$, and $(i_1, \ldots , i_7) := ( 2, 3, 4, 5, 7, 8, 9)$ (we regard the $7$ coefficients $a_2, a_3, a_4, a_5$, $a_7$, $a_8$, $a_9$ as variables).
	In this case we adopt the grevlex order with 
\[
a_8 \prec a_7 \prec a_9 \prec a_3 \prec a_5 \prec a_2 \prec a_4.
\]
	For $\mathbb{F}_{49} [ x, y, z, w]$, we adopt the grevlex order with $w \prec z \prec y \prec x$.
	Put $\mathcal{A} = \mathbb{F}_{49}^{\times} \times \mathbb{F}_{49} \times \mathbb{F}_{49}^{\times}$.
	\item For each $( a_0, a_1, a_6)  \in \mathcal{A}$, as in Case {\bf (N1)} (i) for $q=25$, we enumerate $( a_2, a_3, a_4, a_5, a_7, a_8, a_9) $ such that $C$ is superspecial by Algorithm \ref{alg:enume}.
\end{enumerate}
It follows from the outputs of our computation that there does not exist $( b_1, b_2, a_0, \ldots , a_9 )$ such that $V (P, Q)$ is superspecial.
\end{proof}

The number of possible $(b_1, b_2, a_0, a_1, a_6)$ is $2 \cdot 2 \cdot 48 \cdot 49 \cdot 48 = 451584$.

\begin{rem}
The efficiency of Step (2b) in Main Algorithm in Section \ref{subsec:algorithm} deeply depends on the choice of $s$, $(i_1, \ldots , i_s)$ and the term order.
We experimentally estimated optimal ones for each case to compute.
\end{rem}

\subsection{Our implementation to prove main theorems}\label{subsec:imple}
We implemented computations in the proofs of Propositions \ref{prop:N1(i)q25} -- \ref{prop:Degenerate_q49}, including Algorithm \ref{alg:enume}, in Magma V2.20-10 \cite{Magma, MagmaHP}.
For the source codes and the log files, see the web page of the first author \cite{HPkudo}.

\subsubsection{Timing}

Table \ref{timing} shows the timing of our computation to prove Propositions \ref{prop:N1(i)q25} -- \ref{prop:Degenerate_q49}.
Each entry in the table is the average result of all iterations.

\begin{table}[h]
  \begin{center}
    \caption{Timing data of the computation for Propositions \ref{prop:N1(i)q25} -- \ref{prop:Degenerate_q49} by Algorithm \ref{alg:enume}}\label{timing}
\vspace{0.2cm}
    \scalebox{0.8}{
    \begin{tabular}{c|l|l|c|l|l|l|l|r|r} \hline
    $q$ & quadratic form & Case & $s$ & ~~~~~$t_{mlt}$ & ~~~~~$t_{GBslv}$ & ~~~~~~$t_{sing}$ & ~~~~~$t_{total}$ & Iterations & Total time~~~ \\ \hline
25 & Non-degenerate & {\bf N1} (i) & $8$ & 0.0098007s & 0.0075912s & 0.0074193s & 0.024952s & 3456 & 88.938s \\
 & & {\bf N1} (ii) & $6$ & 0.0047513s & 0.00034230s & 0.0022620s & 0.0074546s & 6912 & 56.610s\\
 & & {\bf N2} & $8$ & 0.035337s & 0.011722s & 0.0088927s & 0.055977s & 2496 & 139.861s\\
 & Degenerate & & $7$ & 0.0076950s & 0.00048167s & 0.0040507s & 0.013324s & 57600 & 783.695s \\ \hline
49 & Non-degenerate & {\bf N1} (i) & $7$ & 0.11507s & 0.048622s & 0.00035776s & 0.19460s & 677376 & 141411.749s \\
     &                             &                     & &                &                &                        &              & & (about 2 days) \\
 & & {\bf N1} (ii) & $6$ & 0.044985s & 0.0099669s & 0.00018157s & 0.055671s & 27648 & 1555.438s \\
 & & {\bf N2} & $7$ & 0.55784s & 0.067364s & 0.00016677s & 0.64785s & 470400 & 309667.773s\\
 & &                 & &              &               &                      &               & & (about 4.5 days) \\
 & Degenerate & & $7$ & 0.086542s & 0.025971s & 0.0013784s & 0.13524s & 451584 & 65566.630s\\ 
 &                     &  & &              &                   &                     &                 & & (about 1 day) \\ \hline
	 \end{tabular}}
 \end{center}
\end{table}

\paragraph{Table Notation}
Let $q$ be the cardinality of $K$.
``$s$'' denotes the number of the indeterminates in the computation to solve a system of algebraic equations in Step (2b) for each case.
``$t_{mlt}$'' denotes the time used in Step (2a) for the computation of the multiple $(P Q)^{p-1}$.
Note that we regard $s$ coefficients of $P$ as indeterminates and thus this computation is done over a multivariate polynomial ring with the indeterminates $x$, $y$, $z$, $w$ whose ground ring is a polynomial ring with $s$ variables.
``$t_{GBslv}$'' is the time used in Step (2b) for solving a system of algebraic equations over $K$ with the Gr\"{o}bner basis computation, and ``$t_{sing}$'' is the cost in Step (2c) for determing whether $C$ is non-singular or not.
``Iterations'' denotes the number of the iterations of Step (2).
``$t_{total}$'' denotes the time used in Steps (2a) -- (2c) for each iteration, wheres ``Total time'' denotes the total time used in Step (2), namely the total time taken for each case.

\paragraph{Workstation.}
We conducted the computation to prove Propositions \ref{prop:N1(i)q25} -- \ref{prop:Degenerate_q49} by a  Windows 8.1 Pro OS, 64 bit computer with 2.60 GHz CPU (Intel Core i5) and 8 GB memory.
In our implementation, the following built-in functions in Magma are called:
\begin{enumerate}
	\item \texttt{GroebnerBasis}: Fix a (decidable) term order $\succ$. (In our computation, let $\succ$ be a grevlex order with some ordering on variables.) For given finite polynomials $f_1, \ldots , f_s$ in $K [ X_1, \ldots , X_n]$ with a (computable) field $K$, this function computes the (reduced) Gr\"{o}bner basis of the ideal $\langle f_1, \ldots , f_s \rangle$ in $K [ X_1, \ldots , X_n]$ with respect to $\succ$ by the $F_4$ algorithm \cite{F4}.
	\item \texttt{Variety}: For given finite polynomials $f_1, \ldots , f_s$ in $K [ X_1, \ldots , X_n]$ with a (computable) field $K$, this function computes the set of all solutions in $K^n$ to $f_i = 0$ for all $1 \leq i \leq s$.
\end{enumerate}

\begin{ex}
We reconsider Example \ref{ex:N2}.
The notation is same as in Example \ref{ex:N2}.
The number of possible $(b_1, b_2, a_1, a_2, a_3 )$, the coefficients which we run, is $2 \cdot 2 \cdot ( 49^2 - 1 ) \times 49 = 470400$.
In our computation over Magma, the average of the time for Step (2b) was $0.067364$ seconds.
The average of the number of the solutions in Step (2b) was $0.091792$, which implies that each system in Step (2b) has no solution or has only the trivial solution $( a_4, a_5, a_6, a_7, a_8, a_9, a_{10} ) = ( 0, \ldots , 0)$ in most cases; in such a case, as this result shows, solving a system of algebraic equations over a finite field with the Gr\"{o}bner basis computation (with respect to a grevlex order) can be much efficient.
Moreover, if the system has no solution, we skip Step (2c).
The average of the time for Step (2c) was $0.00016677$ seconds.
As a result, all the computations for the case of (N2) with $q = 49$ was done in $309667.773$ seconds, which is extremely practical time.
\end{ex}

\subsubsection{Codes for our computation}

\paragraph{Loading our implementation program.}
One can execute our implementation programs on Magma as follows.
Assume here that the file \texttt{code\_q25N1(ii).txt}, which is one of our programs, is in the directory \texttt{C:/Users}. 

\begin{framed}
\vspace{-0.5cm}
\begin{small}
\begin{verbatim}
Magma V2.20-10    Fri Jun 03 2016 00:37:06 on home19890415 [Seed = 1903588510]
Type ? for help.  Type <Ctrl>-D to quit.
> load"C:/Users/code_q25N1(ii).txt";
\end{verbatim}
\end{small}
\vspace{-0.5cm}
\end{framed}

\paragraph{Sample code.}

We here give a piece of our codes.

\begin{framed}
\vspace{-0.5cm}
\begin{small}
\begin{verbatim}
Magma V2.20-10    Sun Jun 05 2016 14:33:33 on home19890415 [Seed = 2004064669]
Type ? for help.  Type <Ctrl>-D to quit.
> p:=5;
> q:=5^2;
> K:=GF(q);
> gen:=Generator(K);
> if gen ne PrimitiveElement(K) then
if>    "gen is not primitive!";
if> end if;
> epsilon:=-gen;
> for e in K do
for>     if e^2 - epsilon eq K!0 then
for|if>         "epsilon is square!";
for|if>     end if;
for> end for;
> s:=6;
> R<[t]>:=PolynomialRing(K,s,"grevlex");
> S<x,y,z,w>:=PolynomialRing(R,4,"grevlex");
> exponents_set:=[
> [ 2*p-2, p-1, p-1, p-1],
> [ 2*p-1, p-2, p-1, p-1],
> [ 2*p-1, p-1, p-2, p-1],
> [ 2*p-1, p-1, p-1, p-2],
> [ p-1, 2*p-2, p-1, p-1],
> [ p-2, 2*p-1, p-1, p-1],
> [ p-1, 2*p-1, p-2, p-1],
> [ p-1, 2*p-1, p-1, p-2],
> [ p-1, p-1, 2*p-2, p-1],
> [ p-2, p-1, 2*p-1, p-1],
> [ p-1, p-2, 2*p-1, p-1],
> [ p-1, p-1, 2*p-1, p-2],
> [ p-1, p-1, p-1, 2*p-2],
> [ p-2, p-1, p-1, 2*p-1],
> [ p-1, p-2, p-1, 2*p-1],
> [ p-1, p-1, p-2, 2*p-1]];
> not_vanished_monomials:=
> {@ x^(E[1])*y^(E[2])*z^(E[3])*w^(E[4]) : E in exponents_set @};
> Coeff_set:=MonomialsOfDegree(R,1);
> a:=[];
> a[1]:= gen^5;
> a[2]:= K!1;
> b:=[];
> b[1]:= K!0;
> b[2]:= - epsilon;
> b[3]:= K!0;
> P:= a[1]*x^2*y + a[2]*x^2*z + b[1]*y^2*z + b[2]*y*z^2 + b[3]*z^2*w;
> Q:= 2*x*w + 2*y*z;
> Mono_set_deg3_unknown:={@ x*y*z, y^2*w, y*z*w, y*w^2, z*w^2, w^3 @};
> for i in [1..#Mono_set_deg3_unknown] do
for>     P:= P + S!(Coeff_set[i])*(Mono_set_deg3_unknown[i]);
for> end for;
> h:=(P*Q)^(p-1);
> F:=[];
> for i in [1..#(not_vanished_monomials)] do
for>     F[i]:=MonomialCoefficient(h,not_vanished_monomials[i]);
for> end for;
> G:=GroebnerBasis(F);
> I:=ideal<R|G>;
> V:=Variety(I);
> #V;
2
> V[1];
<K.1^9, K.1^15, K.1^4, K.1^19, K.1^14, K.1^5>
> V[2];
<K.1^21, K.1^3, K.1^16, K.1^19, K.1^14, K.1^17>
\end{verbatim}
\end{small}
\vspace{-0.5cm}
\end{framed}

In the above piece of code, for Case {\bf (N1)} (ii) with $q=25$ and certain fixed coefficients, we seek the solutions of a system of algebraic equations derived from our criterion to determine whether the Hasse-Witt matrix is zero or not (for the notation, see Propositions \ref{prop:N1(i)q25} and \ref{prop:N1(ii)q25} in Section \ref{subsec:comp_result}).
We can in this case take $\epsilon \notin (\mathbb{F}_{25}^{\times})^2$ so that $- \epsilon$ is a generator of the cyclic group $\mathbb{F}_{25}^{\times}$.
In the above code, \texttt{gen} is the generator of $\mathbb{F}_{25}$ adopted by Magma.
One can verify that it is a generator of the cyclic group $\mathbb{F}_{25}^{\times}$.
For $( b_1, b_2, b_3) = (0, - \epsilon, 0)$ and $(a_1, a_2) = ( (-\epsilon)^5,1)$, we compute the tuples $( a_3, \ldots , a_8) \in (\mathbb{F}_{25})^6$ of the coefficients of $P$ such that the Hasse-Witt matrix of $C = V ( P, Q)$ is zero.
The output shows that the number of solutions $( a_3, \ldots , a_8)$ is $2$, and they are
\begin{eqnarray}
( a_3, a_4, a_5, a_6, a_7, a_8) & = & ( ( - \epsilon )^9, ( - \epsilon )^{15}, ( - \epsilon )^4, ( - \epsilon )^{19}, ( - \epsilon )^{14}, ( - \epsilon )^5 ), \mbox{ and } \nonumber \\
& & ( ( - \epsilon )^{21}, ( - \epsilon )^{3}, ( - \epsilon )^{16}, ( - \epsilon )^{19}, ( - \epsilon )^{14}, ( - \epsilon )^{17} ). \nonumber 
\end{eqnarray}

\section{Automorphism groups}\label{section:6}
In this section, we study the isomorphisms and the automorphism groups of superspecial curves of genus $4$, and enumerate superspecial curves over $\F_{25}$.
We use the notation in Sections \ref{non-degenerate case} and \ref{degenerate case}.
\subsection{Isomorphisms}
Let $K$ be a field.
Let $C_1=V(Q_1,P_1)$ and $C_2=V(Q_2,P_2)$ be curves of genus $4$ over $K$.

If there exists an isomorphism over $K$ from $C_1$ to $C_2$,
the quadratic forms $Q_1$ and $Q_2$ are equivalent over $K$.
Hence it is enough to consider the case of $Q_1=Q_2$, say $Q$. Let $\varphi$ be the symmetric matrix
associated to $Q$.
An isomorphism from $C_1$ to $C_2$ induces an isomorphism
from the space of global sections of the canonical sheaf on $C_2$ to that on $C_1$.
This implies that the set of isomorphisms over $K$ from $C_1$ to $C_2$ is naturally bijective to
\[
\{g \in \tilde \gO_\varphi(K) \mid g P_2 \equiv \lambda P_1 \modulo Q \text{ for some } \lambda \in K^\times\}/\sim
\]
where $g \sim c g$ for some $c\in K^\times$.

\subsection{Automorphisms}
Let $C=V(Q,P)$.
Let $\Aut_K(C)$ denote the group of automorphisms of $C$ over $K$
and write it as $\Aut(C)$ if $K$ is algebraically closed.
We have
\[
\Aut_K(C) = \{g\in \tilde \gO_\varphi(K) \mid g P \equiv \lambda P \modulo Q \text{ for some } \lambda \in K^\times\}/\sim
\]
where $g \sim c g$ for some $c\in K^\times$.

Assume that $Q$ is degenerate. 
Since we have $\tilde \gO_\varphi(K)=K^\times \gO_\varphi(K)$,
putting
\[
G_K:=\{g\in \gO_\varphi(K) \mid g P \equiv \lambda P  \modulo Q \text{ for some } \lambda \in K^\times\},
\]
we have
\[
\Aut_K(C) \ \simeq\  G_K/\{\pm 1_4\}.
\]

\begin{ex}\label{simplest curve}
Let $C$ be the curve defined by
\[
2yw+z^2=0 \quad \text{and}\quad x^3+y^3+w^3=0.
\]
Let $k$ be the algebraic closure of $\F_{25}$.
Using the notation in Section \ref{degenerate case},
we have
\begin{eqnarray*}
G_k &=& \{\diag(d,c,\pm 1,1/c) \mid \lambda=\pm 1, c^3 = \lambda, d^3 = \lambda\}\sqcup\\ 
&&\left\{ \diag(1,1,\pm 1,1)T(c)U(b)sU(a)\diag(d,1,1,1) \left|
\begin{array}{c}
\lambda=\pm 1,\\
a\in \F_{25} \text{ with } (3a^6+1)\ne 0,\\
b=a^{17}+a^{11}+2a^5,\\
c^3=\lambda(3a^6+1),\\
d^3=\lambda
\end{array}
\right.\right\}.
\end{eqnarray*}
To check this, for example use a Gr\"obner basis with indeterminate $\lambda, a,b,c,d$ for the equation $g P \equiv \lambda P  \modulo Q$,
see the web page of the first author \cite{HPkudo},
for a code (by Maple 2016).
Note that $|G| = 2\cdot 2\cdot 3\cdot 3+ 2\cdot 2\cdot (25-6)\cdot 3\cdot 3=720$.
In particular 
$|\Aut(C)| = 720/2=360$.
Moreover we see
that $\Aut(C)$ is isomorphic to the subgroup of $G_k$ consisting of elements with $\lambda = 1$.
\begin{eqnarray*}
\Aut(C) &\simeq& \{\diag(d,c,\pm 1,1/c) \mid c^3 = 1, d^3 = 1\}\sqcup\\ 
&&\left\{ \diag(1,1,\pm 1,1)T(c)U(b)sU(a)\diag(d,1,1,1) \left|
\begin{array}{c}
a\in \F_{25} \text{ with } (3a^6+1)\ne 0,\\
b=a^{17}+a^{11}+2a^5,\\
c^3=3a^6+1,\\
d^3=1
\end{array}
\right.\right\}.
\end{eqnarray*}
Clearly $\Aut(C)$ is decomposed as $\Aut(C) \simeq \mu_3 \times \Aut(C)_{d=1}$,
where $\mu_3$ is the group of cubic roots of one and $\Aut(C)_{d=1}$
is the subgroup of $\Aut(C)$ consisting of elements with $d=1$.
One can check that $\Aut(C)_{d=1}$ is isomorphic to the symmetric group $\fS_5$ of degree $5$.
Indeed $\Aut(C)_{d=1}$ contains
\[
s_1 = \begin{pmatrix}1&&&\\&&&1\\&&1&\\&1&&\end{pmatrix},\quad
s_2=\begin{pmatrix}1&&&\\&&&2-\sqrt{3}\\&&1&\\&2+\sqrt{3}&&\end{pmatrix},
\]
\[
s_3=\begin{pmatrix}1&&&\\ &-1&1&3\\ &1&3&1\\ &3&1&-1 \end{pmatrix},\quad
s_4=\begin{pmatrix}1&&&\\ &1&&\\ &&-1&\\ &&&1\end{pmatrix}
\]
with relations: $s_i^2=1$ and $(s_is_j)^2=1$ if $|i-j|>1$
and $(s_is_j)^2\ne 1$ and $(s_is_j)^3=1$ if $|i-j|=1$.
From this we see that any automorphism of $C$ is defined over $\F_{25}$.
\end{ex}

\begin{prop}
Let $C$ be any superspecial curve of genus $4$ in characteristic $5$. We have
\[
\Aut(C) \simeq \mu_3 \times \fS_5.
\]
\end{prop}
\begin{proof}
In Example \ref{simplest curve} we have seen this for a curve.
The proposition follows from Corollary \ref{MainCorollary}.
\end{proof}

\begin{cor}
There are $21$ $\F_{25}$-isomorphism classes of superspecial curves of genus $4$ over $\F_{25}$.
\end{cor}

\begin{proof}
Let $C$ be the curve as in Example \ref{simplest curve}.
Since the absolute Galois group $\Gamma$ of $\F_{25}$ acts trivially on $\Aut(C)$, we have
\begin{equation}\label{FormsViaAutomorphismGroup}
H^1(\Gamma, \Aut(C)) \simeq \Aut(C)/\text{conjugacy}.
\end{equation}
This set parametrizes the $\F_{25}$-forms of $C$, i.e., superspecial curves of genus $4$ over $\F_{25}$,
thanks to Corollary \ref{MainCorollary}.
The cardinarity of \eqref{FormsViaAutomorphismGroup}
is
$|\mu_3|\cdot |\fS_5/\text{conjugacy}|=21$. 
\end{proof}

It is possible to give a concrete list of 21 curves as in this Corollary:

\begin{ex}\label{CompleteRepresentativesChar5}
Take $\zeta = 1 + \sqrt{3}$ as a generator of $\F_{25}^\times$.
The curves
\begin{eqnarray*}
\text{\rm (I)}\qquad 2yw+z^2=0,& \zeta^i x^3 + \zeta^j y^3 +  w^3=0 & (0\le i \le 2,\quad 0\le j\le 3),\\
\text{\rm (II)}\qquad 2yw+z^2=0,& \zeta^i x^3 + \zeta^k y^3 +  w^3 + zw^2=0 & (0\le i\le 2,\quad k=0,2,3)
\end{eqnarray*}
are complete representatives of $\F_{25}$-isomorphism classes of
superspecial curves of genus $4$ over $\F_{25}$.
The number of $\F_{25}$-rational points on each curve is
\[
\begin{cases}
66 & \text{if}\quad \text{\rm (I)}\ \ (i,j)=(0,0), \\
36 & \text{if}\quad \text{\rm (I)}\ \ (i,j)=(2,1),(1,2),(1,3),(2,3), \\
31 & \text{if}\quad \text{\rm (II)}\ \ (i,k)=(0,2),(2,2), \\
26 & \text{if}\quad \text{\rm (II)}\ \ (i,k) = (0,0),(1,0),(2,0),(0,3),(1,3),(2,3), \\
21 & \text{if}\quad \text{\rm (I)}\ \ (i,j) = (0,1), (1,1), (0,2), (2,2), \\
16 & \text{if}\quad \text{\rm (II)}\ \ (i,k) = (1,2), \\
6  & \text{if}\quad \text{\rm (I)}\ \ (i,j) = (1,0), (2,0), (0,3).
\end{cases}
\]
Indeed, thanks to Theorem \ref{MainTheorem_intro}, considering the transformation $(x\mapsto x, y\mapsto \gamma^2y, z\mapsto \gamma z, w\mapsto w)$ for some $\gamma\in\F_{25}^\times$
and a constant multiplication to the whole of a cubic form, 
any superspecial curve of genus $4$ over $\F_{25}$ is isomorphic to
$V(2yw+z^2,ax^3+by^3+w^3 + czw^2)$ for $a,b\in\F_{25}^\times$ and $c=0,1$.
Taking account of a constant multiplication to $x$, we may assume that $a = \zeta ^i$ with $0\le i\le 2$. See the web page of the first author \cite{HPkudo},
for a code (by Maple 2016) finding the above complete representatives out of
$V(2yw+z^2,\zeta^i x^3+\zeta^jy^3+w^3+czw^2)$ for $0\le i\le 2$, $0\le j\le 23$ and $c=0,1$.
\end{ex}

\subsection{A remark on the mass formula}
Recall the mass formula:
\[
\sum_{(A,\Theta)} \frac{1}{|\Aut(A,\Theta)|}
= \prod_{i=1}^g\frac{(2i-1)!\zeta(2i)}{(2\pi)^{2i}}\prod_{i=1}^g(p^i+(-1)^i)
=  \frac{126139}{21772800},
\]
where $(A,\Theta)$ runs through the isomorphism classes  of principally polarized superspecial abelian varieties of dimension $g=4$ over an algebraically closed field of characteristic $p=5$.
Let $C$ be a(the) superspecial curve of genus $4$ in characteristic $5$. We have
\[
\frac{1}{|\Aut(\Jac(C),\Theta_C)|} =\frac{1}{|\Aut(C)\times\{\pm 1\}|} = \frac{1}{720}. 
\]
The curve occupies about 24\% of the mass.

\begin{prob}
How much of the mass do principally polarized superspecial abelian 4-folds
which are generalized Jacobians of possibly singular curves occupy?
\end{prob}

\appendix

\section{On the radical membership problem over a polynomial ring}\label{sec:radical}

In this appendix, we review a method to solve the \textit{radical membership problem} on polynomial rings over a field $K$ via a theory of Gr\"{o}bner bases.
Let $K$ be a field, and $S = K [ X_1, \ldots , X_n ]$ the polynomial ring with $n$ variables over $K$.
Throughout this appendix, we fix a term order $\succ$ on $S$.

\begin{propp}[\cite{CLO1}, Section 4.2, Proposition 8]\label{lem:radicalmenb}
Let $K$ be an arbitrary field, and $S = K [ X_1, \ldots , X_n ]$ the polynomial ring with $n$ variables over $K$.
For an ideal $I = \langle f_1, \ldots , f_s \rangle_S \subset S$ and a polynomial $f \in S$, $f \in \sqrt{I}$ if and only if $1 \in \widetilde{I}:=  \langle f_1, \ldots , f_s, 1 - Y f \rangle_{S^{\prime}}$, where $S^{\prime} := K [X_1, \ldots , X_n, Y]$.
\end{propp}

Given an ideal $I \subset S$ with its explicit generator and a polynomial $f \in S$, we here give an algorithm to determine whether $f \in \sqrt{I}$ or not.

\begin{algorithm}[H] %
\caption{$\texttt{RadicalMembership} ( f, I )$}
\label{alg:RadicalMenbership}
\begin{algorithmic}[1]
\REQUIRE{A polynomial $f \in S$ and an ideal $I \subset K [X_1, \ldots , X_n]$}
\ENSURE{``$f \in \sqrt{I}$'' or ``$f \notin \sqrt{I}$''}
\STATE $S^{\prime}$ $\leftarrow$ the polynomial ring $K [ X_1, \ldots , X_n , Y]$
\STATE $\succ^{\prime}$ $\leftarrow$ an arbitrary term order on $X_1, \ldots , X_n , Y$
\STATE $\widetilde{I}$ $\leftarrow$ $\langle f_1, \ldots , f_s, 1 - Y f \rangle_{S^{\prime}}$
\STATE Compute the reduced Gr\"{o}bner basis $G$ of $\widetilde{I}$ with respect to $\succ^{\prime}$
\IF{$G = \{ 1 \}$}
	\STATE \textbf{return} ``$f \in \sqrt{I}$''
\ELSE
	\STATE \textbf{return} ``$f \notin \sqrt{I}$''
\ENDIF
\end{algorithmic}
\end{algorithm}

\section{Computing Hasse-Witt matrices of complete intersections}\label{sec:HWgeneral}

In this appendix, we give a method for computing the Hasse-Witt matrix of a curve defined as a complete intersection
via Koszul complex.
This gives a generalization of the method given in Section \ref{subsec:HWgenus4} of this paper.

\subsection{Regular sequences of modules}

This subsection is devoted to a review of some general facts on regular sequences of modules.

\begin{defin}
Let $R$ be a commutative ring with unity, and $M$ an $R$-module.
A sequence $( f_1, \ldots , f_t ) \in R^t$ is said to be an $M$-{\it regular sequence} (or $M$-{\it regular}) if
\begin{enumerate}
\item $M \big/ ( f_1, \ldots , f_t ) M \neq 0$, and
\item For $1 \leq i \leq t$, $f_i$ is a nonzerodivisor in $M \big/ ( f_1, \ldots , f_{i-1} ) M$, i.e.,
there does not exist any $0 \neq x \in M \big/ ( f_1, \ldots , f_{i-1} ) M$ such that $f_i x = 0$.
\end{enumerate}
\end{defin}

\begin{lem}[\cite{Eisenbud_c}, Corollary 17.7]\label{lem:regular1}
Let $R$ be a local commutative ring with unity, and $M$ an $R$-module.
If $\langle f_1, \ldots , f_t \rangle_R \subset R$ is a proper ideal containing an $M$-regular sequence of length $t$, then $( f_1, \ldots , f_t )$ is an $M$-regular sequence.
\end{lem}

\begin{lem}[\cite{Eisenbud_c}, Corollary 17.8]
Let $R$ be a commutative ring with unity, and $M$ an $R$-module.
If $( f_1, \ldots , f_t ) \in R^t$ is an $M$-regular sequence, then $( f_1^n, \ldots , f_t^n )$ is an $M$-regular sequence for any $n > 0$.
\end{lem}

\begin{proof}
We show the statement by the induction on $t$.
Consider the case of $t = 1$.
Let $f  \in R$ be a polynomial such that $M \neq f M$ and $f$ is a nonzerodivisor in $M$.
Obviously we have $M \neq f^n M$.
Assume $f^n x = 0$ in $M$ for some $x \in M$.
Since $f$ is a nonzerodivisor, it follows that $f^{n-1} x$ equals $0 \in M$, and recursively $x = 0$.

Consider the case of $t > 1$.
Since $( f_1^n, \ldots , f_t^n ) M \subset ( f_1, \ldots , f_t ) M$, we have $( f_1^n, \ldots , f_t^n ) M \neq M$.
Here it suffices to show that $f_t$ is a nonzerodivisor in $M \big/ ( f_1^n, \ldots , f_{t-1}^n ) M$.
Indeed, if $f_t$ is a nonzerodivisor in $M \big/ ( f_1^n, \ldots , f_{t-1}^n ) M$ and
if $f_t^n x = 0$ in $M \big/ ( f_1^n, \ldots , f_{t-1}^n ) M$ for some $x \in M$, then $f_t^{n-1} x = 0$ in $M \big/ ( f_1^n, \ldots , f_{t-1}^n ) M$, and recursively $x = 0$ in $M \big/ ( f_1^n, \ldots , f_{t-1}^n ) M$.
Let $P$ be a prime ideal of $R$ with $\mathrm{Ann} (M) \subset P$.
We consider the localization
\[
\left( M \big/ ( f_1^n, \ldots , f_{t-1}^n ) M \right)_P \simeq M_P / ( f_1^n, \ldots , f_{t-1}^n ) M_P \quad \mbox{(as an $R_P$-module)}
\]
at $P$.
Note that if $f_t$ is a nonzerodivisor in $M_P \big/ ( f_1^n, \ldots , f_{t-1}^n ) M_P$, then $f_t$ is a nonzerodivisor in $M \big/ ( f_1^n, \ldots , f_{t-1}^n ) M$.
If there exists $1 \leq i \leq t$ such that $f_i \notin P$, then the either $M_P = ( f_1^n , \ldots , f_{t-1}^n ) M_P$ or $f_t \in ( R_P )^{\times}$, and thus the result holds.
From this, we may assume that $R$ is a local ring and that its maximal ideal contains $f_i$ for all $1 \leq i \leq t$.
The condition that $( f_1, \ldots , f_t )$ is $M$-regular implies that $( f_1, \ldots , f_{t-1}, f_t^n )$ is $M$-regular.
Applying Lemma \ref{lem:regular1}, it is concluded that $( f_t^n, f_1, \ldots , f_{t-1} ) $ is an $M$-regular sequence.
Consequently, repeating the argument, $( f_1^n, \ldots , f_t^n ) $ is an $M$-regular sequence.
\end{proof}


\subsection{Computing Hasse-Witt matrices of complete intersections via Koszul complex}

For non-zero polynomials $f_1, \ldots , f_t \in S = K [ X_0, \ldots , X_r]$,
the Koszul complex $K ( f_1, \ldots , f_t)$ is defined as follows.
For an index $i$, we define the following free $S$-module of rank $\binom{t}{i}$:
\[
K_i ( f_1, \ldots , f_t ) := \bigoplus_{1 \leq j_1 < \cdots < j_i \leq t} S \mathbf{e}_{j_1 \ldots j_i},
\]
where $\mathbf{e}_{j_1, \ldots , j_i}$'s are basis vectors.
We define the homomorphism $\varphi_i : K_i ( f_1, \ldots , f_t ) \longrightarrow K_{i-1} ( f_1, \ldots , f_t )$ by putting
\[
\varphi_i (\mathbf{e}_{j_1 \ldots j_i} ) := \sum_{k=1}^i (-1)^{k-1} f_{j_k} \mathbf{e}_{j_1 \ldots \hat{j_k} \ldots j_i}.
\]
The sequence $K (f_1, \ldots , f_t):=( K_i ( f_1, \ldots , f_t), \varphi_i )_i$ is a complex of free $S$-modules.
Note that any permutation of $f_1, \ldots , f_t$ gives an isomorphic Koszul complex.
It is known that if $( f_1, \ldots , f_t )$ is $S$-regular, then the Koszul complex is exact, that is, it defines a free resolution of $S / \langle f_1 , \ldots , f_t \rangle_S$, see Remark 5.30 in \cite{DL}. 

In this subsection, we give a method to compute the Hasse-Witt matrix of a curve via Koszul complex.
We first define the Koszul complex of {\it graded} free $S$-modules.
For homogeneous polynomials $f_1, \ldots , f_t \in S \smallsetminus \{ 0 \}$
and an index $i$, we define the following graded free $S$-module of rank $\binom{t}{i}$:
\[
K_i ( f_1, \ldots , f_t )_{\mathrm{grd}} := \bigoplus_{1 \leq j_1 < \cdots < j_i \leq t} S (-d_{j_1 \ldots j_i}) \mathbf{e}_{j_1 \ldots j_i},
\]
where we set $d_{j_1 \ldots j_i} := \sum_{k=1}^i \mathrm{deg} ( f_{j_k} )$.
We define the homomorphism $\varphi_i : K_i ( f_1, \ldots , f_t )_{\mathrm{grd}} \longrightarrow K_{i-1} ( f_1, \ldots , f_t )_{\mathrm{grd}}$ by putting
\[
\varphi_i (\mathbf{e}_{j_1 \ldots j_i} ) := \sum_{k=1}^i (-1)^{k-1} f_{j_k} \mathbf{e}_{j_1 \ldots \hat{j_k} \ldots j_i}.
\]
The sequence $K (f_1, \ldots , f_t)_{\mathrm{grd}} :=( K_i ( f_1, \ldots , f_t)_{\mathrm{grd}}, \varphi_i )_i$ is a (co)chain complex of graded free $S$-modules.
As in Section \ref{subsec:HWgenus4}, for $n \geq 1$ we denote by $\varphi_i^{(n)}$ the $i$-th differential of the complex $K ( f_1^n , \ldots , f_t^n)_{\mathrm{grd}}$.
Here we define a homomorphism $\psi_i : K_i ( f_1^n, \ldots , f_t^n )_{\mathrm{grd}} \longrightarrow K_i ( f_1, \ldots , f_t)_{\mathrm{grd}}$ as follows:
\[
\psi_i (\mathbf{e}_{j_1 \ldots j_i} ) := (f_{j_1} \cdots f_{j_i})^{n-1} \mathbf{e}_{j_1 \ldots j_i}.
\]

From now on, we assume that the sequence $( f_1, \ldots , f_t )$ is $S$-regular, and suppose $\mathrm{gcd} (f_i, f_j) = 1$ for $i \neq j$.
To simplify the notation, we set
\[
\begin{split}
M_i^{(n)} := K_i ( f_1^n, \ldots , f_t^n )_{\mathrm{grd}}, \quad \mbox{and} \ I_n := \langle f_1^n, \ldots , f_t^n \rangle_S, \\
M_i := K_i ( f_1, \ldots , f_t )_{\mathrm{grd}}, \quad \mbox{and} \ I := \langle f_1, \ldots , f_t \rangle_S. 
\end{split}
\]
We have the following lemma.

\begin{lem}\label{lem:3}
The following diagram of homomorphisms of graded $S$-modules commutes, and each horizontal sequence is exact:
$$\xymatrix{
0 \ar[r]^{\varphi_{t+1}^{( n )}} & M_t^{(n)}  \ar[d]^{\psi_{t}} \ar[r]^(0.6){\varphi_{t}^{( n )}} \ &  \cdots   \ar[r]^(0.5){\varphi_{2}^{( n )}} &             M_1^{(n)} \ar[d]^{\psi_{1}} \ar[r]^{\varphi_{1}^{( n )}} & M_0^{(n)} = S  \ar[d]^{\psi_{0}} \ar[r]^(0.5){\varphi_0^{( n )}} & M_{-1}^{(n)} := S / I_n                  \ar[d]^{\psi} \ar[r]    & 0            \\
0 \ar[r]^{\varphi_{t+1}^{( 1 )}} & M_t                                    \ar[r]^(0.6){\varphi_t^{( 1 )}} & \cdots                                   \ar[r]^(0.5){\varphi_2^{( 1 )}} & M_1  \ar[r]^{\varphi_{1}^{( 1 )}} & M_0 = S         \ar[r]^(0.5){\varphi_0^{( 1 )}} & M_{-1} := S / I     \ar[r]    & 0            
}$$
where $\psi_0$ is the identity map on $S$, and $\psi$ is the homomorphism defined by $h + I_n \mapsto h + I$.
\end{lem}

Let $K$ be a perfect field with $\mathrm{char} ( K )  = p > 0$.
Let $\mathbf{P}^r = \mathrm{Proj} ( S )$ denote the projective $r$-space for the polynomial ring $S := K [ X_0, \ldots, X_r]$.
For a graded $S$-module $M$, let $\widetilde{M}$ denote the sheaf associated with $M$ on $\mathbf{P}^r$.
Now we describe a method to compute the Hasse-Witt matrix of the curve $C = V ( f_1, \ldots , f_{r-1} ) \subset \mathbf{P}^r$ for given $p$ and $r-1$ homogeneous polynomials $f_1, \ldots, f_{r-1} \in S$ with $\mathrm{gcd} ( f_i, f_j) = 1$ in $S$ for $1 \leq i < j \leq r-1$ such that $( f_1, \ldots , f_{r-1})$ is $S$-regular.
Note that such a $C$ is said to be a complete intersection in $\mathbf{P}^r$.
We use the same notation as in Lemma \ref{lem:3}, and take $n = p$.
Put $\varphi_i := \varphi_i^{(1)}$ and
\begin{equation}
\Phi_i:=\widetilde{\varphi_i},\quad \Phi_i^{(p)}:= \widetilde{\varphi_i^{(p)}},\quad \Psi:=\widetilde{\psi},\quad \mbox{and}\quad \Psi_i:=\widetilde{\psi_i}.
\end{equation}
By Lemma \ref{lem:3}, the following diagram commutes:
$$\xymatrix{
H^1 \left( C, {\cal O}_C \right)                  \ar[dd]_{F^{\ast}} \ar[rr]^(0.5){\cong} \ar[rd]^{{(F_1 |_{C^p})}^{\ast}} &  & H^2 ( \mathbf{P}^r, \widetilde{I} )     \ar[r]^(0.45){\cong}              \ar[d]^{F_1^{\ast}}  & \mathrm{Ker} \left( H^r ( \Phi_{r-1} ) \right) \ar[d]^{F_1^{\ast}} \\
& H^1 \left( C^p, {\cal O}_{C^p} \right) \ar[r]^{\cong} \ar[ld]^{H^1 ( \Psi )} & H^2 ( \mathbf{P}^r, \widetilde{I_p} ) \ar[r]^(0.45){\cong} \ar[d]^{H^2 ( \Psi_0 )} & \mathrm{Ker} \left( H^r ( \Phi_{r-1}^{( p )} ) \right) \ar[d]^{H^r ( \Psi_{r-1} )} \\
H^1 \left( C, {\cal O}_C \right)                  \ar[rr]^(0.5){\cong} &  & H^2 ( \mathbf{P}^r, \widetilde{I} ) \ar[r]^(0.45){\cong} & \mathrm{Ker} \left( H^r ( \Phi_{r-1} )  \right) 
}$$
where $F_1$ (resp. $F$) is the Frobenius morphism on $\mathbf{P}^r$ (resp. $C$) and $C^p :=  V \left( f_1^p, \ldots, f_{r-1}^p \right)$.
In a similar way to the proof of Proposition 3.1.4 in this paper, we have the following proposition.

\begin{prop}\label{prop:HW_more_general}
Let $K$ be a perfect field with $\mathrm{char} ( K )  = p > 0$.
Let $f_1, \ldots , f_{r-1}$ be homogeneous polynomials with $d_{j_1 \ldots j_{r-2}} \leq r$ for all $1 \leq j_1 < \cdots < j_{r-2} \leq r-1$ such that $\mathrm{gcd}( f_i, f_j ) = 1$ in $S:=K [ X_0, \ldots , X_r]$ for $i \neq j$. 
Suppose that $( f_1, \ldots , f_{r-1})$ is an $S$-regular sequence.
Let $C = V ( f_1, \ldots , f_{r-1} )$ be the curve defined by the equations $f_1 = 0, \ldots ,f_{r-1} = 0$ in $\mathbf{P}^r$.
Write $(f_1 \cdots f_{r-1} )^{p-1} =  \sum c_{i_0, \ldots , i_{r}} X_0^{i_0} \cdots X_r^{i_r}$ and
\begin{equation}
\{ (k_0, \ldots , k_r) \in ( \mathbb{Z}_{<0} )^{r+1} : \sum_{i=0}^r k_i  = - \sum_{j=1}^{r-1} \mathrm{deg} ( f_j ) \} = \{ (k_0^{(1)}, \ldots , k_r^{(1)}), \ldots , (k_0^{(g)}, \ldots , k_r^{(g)} ) \}, \nonumber
\end{equation}
where we note that $g = \mathrm{dim}_K H^1 (C, \mathcal{O}_C)$.
Then the Hasse-Witt matrix of $C$ is given by
\begin{equation}
\left[
\begin{array}{ccc}
	c_{- k_0^{(1)} p + k_0^{(1)}, \ldots , - k_r^{(1)} p + k_r^{(1)}} & \cdots & c_{- k_0^{(g)} p + k_0^{(1)}, \ldots , - k_r^{(g)} p + k_r^{(1)}} \\
\vdots & & \vdots \\
	c_{- k_0^{(1)} p + k_0^{(g)}, \ldots , - k_r^{(1)} p + k_r^{(g)}} & \cdots & c_{- k_0^{(g)} p + k_0^{(g)}, \ldots , - k_r^{(g)} p + k_r^{(g)}}
\end{array}
\right]. \nonumber 
\end{equation}
\end{prop}

%
%
%
%
%

\end{document}